\documentclass[11pt, notitlepage]{article}
\usepackage{amssymb,amsmath,comment}
\catcode`\@=11 \@addtoreset{equation}{section}
\def\thesection{\arabic{section}}

\def\theequation{\thesection.\arabic{equation}}
\catcode`\@=12
\usepackage{colortbl}%
\usepackage{mathtools}
\usepackage{a4wide}

\newcommand{\ds} {\displaystyle}
\newcommand{\e}{\epsilon}

\newcommand{\Om} {\Omega}
\newcommand{\ra} {\rightarrow}

\newcommand{\De} {\Delta}
\newcommand{\la} {\lambda}

\newcommand{\noi} {\noindent}

\newcommand{\mb} {\mathbb}
\newcommand{\mc} {\mathcal}

\usepackage[all]{xy}
\catcode`\@=11
\def\theequation{\@arabic{\c@section}.\@arabic{\c@equation}}
\catcode`\@=12

\def\QED{\hfill {$\square$}\goodbreak \medskip}

\newtheorem{Theorem}{Theorem}[section]
\newtheorem{Lemma}[Theorem]{Lemma}
\newtheorem{Proposition}[Theorem]{Proposition}

\newtheorem{Remark}[Theorem]{Remark}
\newtheorem{Definition}[Theorem]{Definition}

\begin{document}

\title
{ Existence and stabilization results for a singular parabolic equation involving the fractional Laplacian }

\author{J. Giacomoni\footnote{LMAP (UMR CNRS 5142) Bat. IPRA,
  Avenue de l'Universit\'e
   F-64013 Pau, France. email:jacques.giacomoni@univ-pau.fr}, ~~ T. Mukherjee\footnote{Department of Mathematics, Indian Institute of Technology Delhi,
Hauz Khaz, New Delhi-110016, India.
 e-mail: tulimukh@gmail.com}~ and ~K. Sreenadh\footnote{Department of Mathematics, Indian Institute of Technology Delhi,
Hauz Khaz, New Delhi-110016, India.
 e-mail: sreenadh@gmail.com} }

\date{}

\maketitle

\begin{abstract}

In this article, we study the following parabolic equation involving the fractional Laplacian with singular nonlinearity
\begin{equation*}
 \quad (P_{t}^s) \left\{
\begin{split}
 \quad u_t + (-\De)^s u &=  u^{-q} + f(x,u), \;u >0\; \text{in}\;
(0,T) \times \Om, \\
  u &= 0 \; \mbox{in}\; (0,T) \times (\mb R^n \setminus\Om),\\
 \quad \quad \quad \quad u(0,x)&=u_0(x) \; \mbox{in} \; {\mb R^n},
\end{split}
\quad \right.
\end{equation*}
where  $\Om$ is a bounded domain in $\mb{R}^n$ with smooth boundary $\partial \Om$, $n> 2s, \;s \in (0,1)$, $q>0$,
${q(2s-1)<(2s+1)}$, $u_0 \in L^\infty(\Om)\cap X_0(\Om)$  and $T>0$. We suppose that the map $(x,y)\in \Om \times \mb R^+ \mapsto f(x,y)$ is a bounded below Carath\'{e}odary function, locally Lipschitz with respect to second variable and {uniformly} for $x \in \Om$ it satisfies
 \begin{equation}\label{cond_on_f}
{ \limsup_{y \to +\infty} \frac{f(x,y)}{y}<\la_1^s(\Om)},
 \end{equation}
 where $\la_1^s(\Om)$ is the first eigenvalue of $(-\De)^s$ in $\Om$ with homogeneous Dirichlet boundary condition in $\mb R^n \setminus \Om$. We prove the existence and uniqueness of weak solution to $(P_t^s)$ on assuming $u_0$ satisfies an appropriate cone condition. We use the semi-discretization in time with implicit Euler method and study the stationary problem to prove our results. {We also show additional } regularity on the solution of $(P_t^s)$ when we regularize our initial function $u_0$.
\medskip

\noi \textbf{Key words:} Non-local operator, fractional Laplacian, singular nonlinearity, parabolic equation.

\medskip

\noi \textit{2010 Mathematics Subject Classification:} 35R11, 35R09,
35A15.

\end{abstract}

\section{Introduction}
In this paper, we study the existence and uniqueness of weak solution for the following fractional parabolic equation with singular nonlinearity
\begin{equation*}
 \quad (P_{t}^s) \left\{
\begin{split}
 \quad u_t + (-\De)^s u &=  u^{-q} + f(x,u), \;u >0\; \text{in}\;
\Lambda_T, \\
  u &= 0 \; \mbox{in}\; \Gamma_T,\\
 u(0,x)&=u_0(x) \; \mbox{in} \; {\mb R^n},
\end{split}
\quad \right.
\end{equation*}
where  $\Lambda_T = (0,T) \times \Om$, $\Gamma_T= (0,T) \times (\mb R^n \setminus \Om)$, $\Om$ is a bounded domain in $\mb{R}^n$ with smooth boundary $\partial \Om$ (atleast $C^2$), $n>2s,\; s \in (0,1),\; q>0, \; q(2s-1)<(2s+1)$ and $T>0$. The map $(x,y)\in \Om \times \mb R^n \mapsto f(x,y)$ is assumed to be a bounded below Carath\'{e}odary function, locally Lipschitz with respect to second variable and {uniformly} for $x \in \Om$ it satisfies
 $$ { \limsup_{y \to +\infty} \frac{f(x,y)}{y}<\la_1^s(\Om)}, $$
 where $\la_1^s(\Om)$ is the first eigenvalue of $(-\De)^s$ in $\Om$ with (homogeneous) Dirichlet boundary condition in $\mb R^n \setminus \Om$. The fractional Laplace operator $(-\De)^s$ is defined as
$$ (-\De)^s u(x) = 2C^s_n\mathrm{P.V.}\int_{\mb R^n} \frac{u(x)-u(y)}{\vert x-y\vert^{n+2s}}\,\mathrm{d}y$$
{where $\mathrm{P.V.}$ denotes the Cauchy principal value and $C^s_n=\pi^{-\frac{n}{2}}2^{2s-1}s\frac{\Gamma(\frac{n+2s}{2})}{\Gamma(1-s)}$, $\Gamma$ being the Gamma function.

\noi In this article,
we will be concerned with the  nonlocal  problem $(P^s_t)$  that {involves the}
fractional Laplacian. A large variety of diffusive problems in Physics  are satisfactorily described by the classical Heat
equation.  However, anomalous diffusion that follow non-Brownian scaling  is nowadays intensively
studied  with wide range of  applications  in  physics, finance, biology and many others. The  governing equations of such mathematical models involve the fractional Laplacian.  For a detailed survey on this we refer to  \cite{vazquez1, vazquez2} and references therein. It is natural to study the {local and global} existence and stabilization results for such problems. \\

\noi {Singular parabolic problems in the local case has been studied by authors in \cite{fm,dm,ags}. The inspiring point for us was the work of M. Badra et al. \cite{bbg}, where the  existence and stabilization results for parabolic problems where the principal part of the equation is the $p$-Laplacian operator, has been studied  when $0<q < 2+ \frac{1}{p-1}$. In \cite{bg} Bougherara and Giacomoni authors proved the existence of unique mild solution to  the problem for all $q>0$ when $u_0 \in (C_0(\overline{\Omega}))^+$. In the present  work we extend the results obtained in \cite{bbg} to the non-local case. However, there is a substantial difference between local and nonlocal operators. This difference is reflected in the way of  construction of sub and super solutions of stationary problems associated to $(P^s_t)$ as well as the validity of the weak comparison principle. Nonethless, we will show that the semi-discretization  in time method used in \cite{bbg} can still be efficient in this case.  }\\

 \noi {Coming to the non-local case, singular elliptic equations involving fractional laplacian has been studied by Barios et al.  in \cite{bbmp} and Giacomoni et al. in \cite{TJS}. More specifically, existence and multiplicity results for the equation
 \[(-\De)^su = \la u^{-q}+ u^p \; \text{in}\; \Om,\; u=0 \; \text{in}\; \mb R^n\setminus \Om\]
 has been shown for $0<q\leq1$ and $0<p< 2^*_s-1$ where $2^*_s= \displaystyle\frac{2n}{n-2s}$ in \cite{bbmp} and $p=2^*_s-1$ in \cite{TS1}. Whereas the case $q>0$ and $p= 2^*_s-1$ has been studied in \cite{TJS}. {Concerning the parabolic problems involving the fractional laplacian, we cite \cite{vazquez1,vazquez2,ai,fk} and references therein. Caffarelli and Figalli studied the regularity of solutions to fractional parabolic obstacle problem in \cite{cf}. In \cite{kimlee}, authors studied the  H\"{o}lder estimates for singular problems of the type $(-\De )^s u^m +u_t=0$ where $\frac{n-2s}{n+2s}<m<1.$ In \cite{tp},  the summability of solutions with respect to the summability of the data is studied. In \cite{ap}, authors studied  the influence of Hardy potential on the existence and nonexistence  of positive solutions for fractional heat equation.   To the best of our knowledge there are no works  on parabolic equations  with fractional laplacian and singular nonlinearity.} \\

\noi { In this work, we first define the positive cone motivated from the work of \cite{aJs} and obtain the existence of solutions in this cone for  the
 elliptic  problem $(S)$ in section 2 associated to the semi-discretization of  $(P^s_t).$   Using this, we proved the existence and uniqueness of
 solution and its regularity for the parabolic problem (see $(G^s_t)$ in section 2 with bounded source term $h(x,t)$ and principal diffusion operator $(-\De)^s  - u^{-q}$  in section $4$). Finally using the new uniqueness results for the stationary problem  proved in section $5$, we prove the existence and uniqueness of solutions to the problem $(P^s_t)$ in section $6$. Thanks to nonlinear accretive operators theory, we also find that these solutions  are more regular when the regularity assumption is refined on the initial condition. We end our paper by showing that the solution to $(P^s_t)$ converges to the unique solution of its stationary problem as $t \to \infty$ in section $7$. In this aim, we extend existence and regularity results about the stationary problem proved in  \cite{aJs}.}

\section{Functional Setting and Main results}
We denote the usual fractional Sobolev space by $H^s(\Om)$ endowed with the  Gagliardo norm
\[\|u\|_{H^s(\Om)}= \|u\|_{L^2(\Om)}+ \left(\int_\Om \int_\Om \frac{|u(x)-u(y)|^2}{|x-y|^{n+2s}}~dxdy\right)^{\frac12}. \]
Then we consider the following space
\[X(\Om)= \left\{u|\;u:\mb R^n \ra\mb R \;\text{is measurable},
u|_{\Om} \in L^2(\Om)\;
 \text{and}\;  \frac{\left(u(x)- u(y)\right)}{ |x-y|^{\frac{n+2s}{2}}}\in
L^2(Q)\right\},\]
\noi where $Q=\mb R^{2n}\setminus(\mc C\Om\times \mc C\Om)$ and
 $\mc C\Om := \mb R^n\setminus\Om$. The space $X(\Om)$ is endowed with the norm defined as
\begin{align*}
 \|u\|_{X(\Om)} = \|u\|_{L^2(\Om)} +\left( \int_{Q}\frac{|u(x)-u(y)|^{2}}{|x-y|^{n+2s}}dx
dy\right)^{\frac12}.
\end{align*}
Now we define the space
 $ X_0(\Om) = \{u\in X(\Om) : u = 0 \;\text{a.e. in}\; \mb R^n\setminus \Om\}$
equipped with the norm
\[\|u\|_{X_0(\Om)}=\left( C_n^s\int_{Q}\frac{|u(x)-u(y)|^{2}}{|x-y|^{n+2s}}dx
dy\right)^{\frac12}\]
where $C_n^s$ is defined in section $1$ and {it is well known that} $X_0(\Om)$ forms a Hilbert space with this norm {(see \cite{fraceigen})}.
From the embedding results, we know that $X_0(\Om)$ is continuously and compactly embedded in $L^r(\Om)$  when $1\leq r < 2^*_s$ and the embedding is continuous but not compact if $r= 2^*_s$. For each $\alpha \geq 0$, we set
\[ C_{\alpha} = \sup \left \{ \int_{\Om} |u|^\alpha dx : \|u\|_{X_0(\Om)} = 1\right \}.\]
Then $C_0= |\Om|=$ Lebesgue measure of $\Om$ and $\int_{\Om} |u|^\alpha dx \leq C_\alpha \|u\|^{\alpha}$, for all $ u \in X_0(\Om)$. Let us consider a more general problem
\begin{equation*}
(G_t^s)\left\{
\begin{split}
 \quad u_t + (-\De)^s u &=  u^{-q} + h(t,x) ,\;u>0\; \text{in}\;
\Lambda_T, \\
  u &= 0 \; \mbox{in}\; \Gamma_T,\\
 u(0,x)&=u_0(x) \; \mbox{in} \; {\mb R^n},
\end{split}
\right.
\end{equation*}
where $T>0$, $s\in(0,1)$, $h \in L^\infty(\Lambda_T)$, $q>0$, $q(2s-1)<(2s+1)$ and $u_0 \in L^\infty(\Om) \cap X_0(\Om)$.
In order to define weak solution for the problem $(G_t^s)$, we need to introduce the following space
\[\mc A(\Lambda_T):= \{u : \; u \in L^\infty(\Lambda_T), \; u_t \in L^2(\Lambda_T), \; u \in L^\infty(0,T;X_0(\Om))\}. \]
 We have the following result as a direct consequence of  Aubin-Lions-Simon Lemma (see \cite{aubin-simon}).
\begin{Lemma}\label{AinC}
Suppose $u \in L^\infty(0,T;X_0(\Om))$ and $u_t \in L^2(\Lambda_T)$. Then $u \in C([0,T]; L^2(\Om))$ and the embedding is compact.
\end{Lemma}
We now define the notion of weak solution for the problem $(G_t^s)$.
\begin{Definition}
We say $u \in \mc A(\Lambda_T)$ is a weak solution of $(G_t^s)$ if
\begin{enumerate}
\item[1.] for any compact subset $K \subset \Lambda_T$, $ess\inf_K u >0$,
\item[2.] for every $\phi \in \mc A(\Lambda_T)$,
\[\int_{\Lambda_T}\frac{\partial u}{\partial t}\phi~dxdt + C_n^s\int_0^T\int_{Q}\frac{(u(x)-u(y))(\phi(x)-\phi(y))}{|x-y|^{n+2s}}dydxdt
  =\int_{\Lambda_T}( u^{-q}+ h(t,x))\phi dxdt,\]
\item[3.] $u(0,x)= u_0(x)$ a.e. in $\Om$.
\end{enumerate}
\end{Definition}
We remark that because of Lemma \ref{AinC}, we get $\mc A(\Lambda_T) \subset C([0.T]; L^2(\Om))$ which means that the third point of the above definition makes sense.\\
Now, we define a conical shell $\mc C$ as the set of functions $v \in L^\infty(\Om)$ such that there exist constants $k_1,k_2>0$ such that
\begin{equation*}
\left\{
\begin{split}
&k_1\delta^s(x)\leq v \leq k_2\delta^s(x) & \text{ if } q<1,\\
&k_1 \delta^s(x) \left(\ln \left( {\frac{r}{\delta^s(x)}}\right)\right)^{\frac{1}{2}} \leq v \leq k_2 \delta^s(x) \left(\ln \left( {\frac{r}{\delta^s(x)}}\right)\right)^{\frac{1}{2}}  & \text{ if } q=1,\\
&k_1\delta^{\frac{2s}{q+1}}(x) \leq v \leq k_2 \delta^{\frac{2s}{q+1}}(x)  & \text{ if } q>1,
\end{split}
\right.
\end{equation*}
where $\delta(x) := \text{dist}(x,\partial \Om)$ for $x \in \Om$ {and $r>\text{diam}(\Om)$}. We set
\[C_0(\overline \Om) := \left\{ u \in C(\overline \Om): \; u =0 \;\text{on} \; \partial \Om\right\}.\]
We begin by considering the stationary problem $(S)$:
\begin{equation*}
 \quad (S) \left\{
\begin{split}
 \quad u + \la \left((-\De)^s u -  u^{-q}\right) &= g, \;\; \; u>0\; \text{in}\;\Om, \\
 \quad \quad \quad \quad u &= 0 \; \text{in}\; \mb R^n \setminus \Om,\\
 \end{split}
\quad \right.
\end{equation*}
where $g \in L^\infty(\Om)$ and $\la >0$ is a real parameter. The notion of weak solution is defined as follows.
\begin{Definition}
We say $ u \in X_0(\Om)$ is a weak solution of $(S)$ if
\begin{enumerate}
\item[1.] for any compact subset $K \subset \Om$, $ess\inf_K u >0$,
\item[2.] for every $\phi \in X_0(\Om)$,
\[\int_\Om u\phi~dx + \la \left( C_n^s \int_Q\frac{(u(x)-u(y))(\phi(x)-\phi(y))}{|x-y|^{n+2s}}~dxdy-\int_{\Om} u^{-q}\phi~dx  \right)= \int_{\Om}g\phi~dx.  \]
\end{enumerate}
\end{Definition}

We prove the following theorem considering the problem $(S)$.
\begin{Theorem}\label{mt1}
If $g \in L^\infty(\Om)$, $q>0$ and ${q}(2s-1)<(2s+1)$, then for any $\la>0$, problem $(S)$ has a unique weak solution $u_\la \in X_0(\Om)\cap  {\mc C \cap C^\alpha(\mb R^n)}$ where $\alpha = s$ if $q<1$, $\alpha=s-\e$ if $q=1$, for any $\e>0$ small enough and $\alpha = \displaystyle\frac{2s}{q+1}$ if $q>1$.
\end{Theorem}

\noi In the case $q(2s-1)\geq (2s+1)$, we get less regularity on solution of $(S)$. So we will have a weaker notion of solution in this case for which we define the set
\[\Theta:= \{\phi:\;\phi: \mb R^n \to \mb R \; \text{measurable and}\; (-\De)^s\phi \in L^\infty(\Om),\; \phi\equiv 0 \; \text{on}\; \mb R^n \setminus \Om^\prime,\; \Om^\prime \Subset \Om \}.\]

\begin{Theorem}\label{geq-case}
Let $g \in L^\infty(\Om)$, $q>1$ and $q(2s-1) \geq (2s+1)$ then for any $\la>0$, there exists a $u_\la \in L^1(\mb R^n)$ satisfying $u\equiv 0$ in $\mb R^n \setminus \Om$, $\inf\limits_K u_\la >0$ for every $K \Subset \Om$  and
\[\int_\Om u_\la\phi~dx + \la \left( C_n^s \int_Q\frac{(u_\la(x)-u_\la(y))(\phi(x)-\phi(y))}{|x-y|^{n+2s}}~dxdy-\int_{\Om} u_\la^{-q}\phi~dx  \right)= \int_{\Om}g\phi~dx  \]
for any $\phi \in \Theta $. Moreover $u^{\beta}_\la \in X_0(\Om)$ where $\beta > \max \left\{1, \left(1-\frac{1}{2s} \right)\left( \frac{q+1}{2}\right) \right\}$ but $u_\la \notin X_0(\Om)$.

\end{Theorem}

\begin{Definition}
We say that ${u(t)} \in \mc C$ uniformly for each $t \in [0,T]$ when there exist $\psi_1, \psi_2 \in \mc C$ such that $\psi_1(x) \leq u(t,x) \leq \psi_2(x)$ a.e. $(t,x)\in [0,T] \times \Om $.
\end{Definition}
We prove the following existence and uniqueness result for the problem $(G^s_t)$ using semi-discretization in time with implicit Euler method, Theorem \ref{mt1}, energy estimates and the {weak comparison principle}.
\begin{Theorem}\label{mt2}
If $h(t,x) \in L^{\infty}(\Lambda_T)$, $u_0\in X_0(\Om)\cap \mc C$, $q>0$ and $q(2s-1)<(2s+1)$, then there exists a unique weak solution $u \in C([0,T]; X_0(\Om))$ for the problem $(G^s_t)$ such that ${u(t)} \in \mc C$ uniformly for each $t \in [0,T]$. Also, $u$ satisfies
\begin{equation}\label{mt2_eq}
\begin{split}
\int_0^{t} \int_{\Om} \left( \frac{\partial u}{\partial t}\right)^2 dxd\tau & + \frac12 \|u(t,x)\|_{X_0(\Om)}^2- \frac{1}{1-q}\int_{\Om} u^{1-q}(t,x)dx \\
&\quad = \int_0^{t} \int_{\Om} h(\tau, x)\frac{\partial u}{\partial t} dxd\tau + \frac12 \|u_0(x)\|_{X_0(\Om)}^2- \frac{1}{1-q}\int_{\Om} u_0^{1-q}(x)dx
\end{split}
\end{equation}
for any $t \in [0,T]$.
\end{Theorem}
The solution obtained in above theorem can be shown to be more regular under some extra assumptions as can be seen in the next result.
\begin{Proposition}\label{para_sing_prop1}
Under the hypothesis of Theorem \ref{mt2}, if $u_0 \in \overline{\mc D(L)}^{L^{\infty}(\Om)}$, where
\[{{\mc D(L)}} := \{ v \in \mc C \cap X_0(\Om): \; L(v):= (-\De)^sv -v^{-q} \in L^{\infty}(\Om)\}\]
then the solution to $(G^s_t)$ obtained in Theorem \ref{mt2} belongs to $C([0,T];C_0(\overline \Om))$. Also $u$ satisfies:
\begin{enumerate}
\item If $v$ is another solution to $(G^s_t)$ with initial condition $v_0 \in \overline{\mc D(L)}^{L^\infty(\Om)}$ and nonhomogenous term $b \in L^\infty(\Lambda_T)$, then for any $t \in [0,T]$,
    \[\|u(t,\cdot)- v(t,\cdot)\|_{L^\infty(\Om)} \leq \|u_0-v_0\|_{L^\infty(\Om)}+ \int_0^t \|h(\tau,\cdot)- b(\tau,\cdot) \|_{L^\infty(\Om)}d\tau.\]

\item If $u_0 \in \mc D(L)$ and $h \in W^{1,1}(0,T;L^\infty(\Om))$, then $u \in W^{1,\infty}(0,T; L^\infty(\Om))$, $(-\De)^su + u^{-q} \in L^\infty(\Lambda_T)$ and the following holds true for any $t \in [0,T]$,
    \[\left\|\frac{du(t,\cdot)}{dt}\right\|_{L^\infty(\Om)} \leq \|(-\De)^su_0 + u^{-q}_0+h(0,\cdot) \|_{L^\infty(\Om)}+ \int_0^T \left\|\frac{dh(\tau,\cdot)}{dt}\right\|_{L^\infty(\Om)}d\tau. \]
\end{enumerate}
\end{Proposition}
 In order to establish Theorem \ref{mt3}, we need the following result.

\begin{Theorem}\label{mt4}
Suppose $q>0$, $q(2s-1)<(2s+1)$ and $f : \Om \times \mb R^+ \to \mb R$ be bounded below Carath{\'{e}}odary function satisfying \eqref{cond_on_f}. {Assume that} $f$ is locally Lipschitz with respect to second variable uniformly in $\Om$ and $\frac{f(x,y)}{y}$ is decreasing in $\mb R^+$ for a.e. $x \in \Om$. Then the following problem $(Q^s)$ has a unique solution $\hat u \in {X_0(\Om)\cap {\mc C}} \cap C^\alpha(\mb R^n)$ where $\alpha = s$ if $q<1$, $\alpha=s-\e$ if $q=1$, for any $\e>0$ small enough and $\alpha = \displaystyle\frac{2s}{q+1}$ if $q>1$. :
\begin{equation*}
(Q^s)\left\{
\begin{split}
(-\De)^s \hat u -{\hat u}^{-q} &= f(x, \hat u)\; \text{in} \; \Om,\\
 \hat u &= 0 \; \text{in} \; \mb R^n \setminus \Om.
\end{split}
\right.
\end{equation*}
\end{Theorem}

Coming back to our original problem $(P^s_t)$, we have the following theorem {:}
\begin{Theorem}\label{mt3}
Assume $q>0$, $q(2s-1)<(2s+1)$ and $f(t,x)$ to be a bounded below Carath\'{e}odory function, locally Lipschitz with respect to second variable uniformly in $x \in \Om$ and satisfies \eqref{cond_on_f}. If $u_0 \in X_0(\Om)\cap \mc C$, then for any $T>0$, there exists a unique weak solution $u$ to $(P^s_t)$ such that ${u(t)} \in \mc C$ uniformly for $t \in [0,T]$  and $u \in C([0,T]; X_0(\Om))$. Moreover for any $t \in [0,T]$,
\begin{equation*}
\begin{split}
&\int_0^{t} \int_{\Om} \left( \frac{\partial u}{\partial t}\right)^2 dxd\tau + \frac12  \|u(t,x)\|_{X_0(\Om)}^2- \frac{1}{1-q}\int_{\Om} u^{1-q}(t,x)dx \\
&\quad = \int_\Om F(x,u(t)) dx + \frac12 \|u_0(x)\|_{X_0(\Om)}^2- \frac{1}{1-q}\int_{\Om} u_0^{1-q}(x)dx- \int_\Om F(x,u_0)dx,
\end{split}
\end{equation*}
where $F(x,z):= \int_0^z f(x,z)dz$.
\end{Theorem}

Using Proposition \ref{para_sing_prop1}, on a similar note we have the following proposition regarding the solution of  problem $(P_t^s)$.

\begin{Proposition}\label{para_sing_prop2}
Assume that the hypothesis of Theorem \ref{mt3} are true. If $u_0 \in \overline{\mc D(L)}^{L^\infty(\Om)}$, then the solution of $(P_t^s)$ belongs to $C([0,T];C_0(\overline{\Om}))$. Let $\alpha\geq0$ denotes the Lipschitz constant of $f(\cdot,x)$ in $[\underline{u},\overline{u}]$, where $\underline{u}$ and $\overline{u}$ denotes the sub and super solution respectively of $(Q^s)$, then the following holds:
\begin{enumerate}
\item If $v$ is another weak solution to $(P_t^s)$ with initial condition $v_0 \in \overline{\mc D(L)}^{L^\infty(\Om)} $, then
    \[\|u(t,\cdot)-v(t,\cdot)\|_{L^\infty(\Om)} \leq \exp(\alpha t) \|u_0-v_0\|_{L^\infty(\Om)},\; 0\leq t\leq T.\]
    \item If $u_0 \in \mc D(L)$, then $u \in W^{1,\infty}([0,T];L^\infty(\Om))$ and $(-\De)^s u + u^{-q} \in L^\infty(\Lambda_T)$. Also the following holds:
        \[\left\|\frac{du(t,\cdot)}{dt}\right\|_{L^\infty(\Om)} \leq \exp(\alpha t)\|(-\De)^su_0 + u^{-q}_0+f(x,u_0) \|_{L^\infty(\Om)}.\]
\end{enumerate}
\end{Proposition}
Finally, we can show the following asymptotic behavior of solutions of $(P_t^s)$.
\begin{Theorem}\label{mt5}
Under the hypothesis of Theorem \ref{mt3} and the assumption that $y \mapsto \frac{f(x,y)}{y}$ is decreasing in $(0,\infty)$ a.e. $x \in \Om$, the solutions to $(P_t^s)$ is defined in $(0,\infty)\times \Om$ and it satisfies
\[u(t) \to \hat u \text{ in } L^\infty(\Om) \text{ as } t \to \infty,\]
where $\hat u$ is defined in Theorem \ref{mt4}.
\end{Theorem}

\begin{Remark}
We can conclude the results for the problem $(P_t^s)$ in a similar manner when $q>-1$ and $q(2s-1)<(2s+1)$ holds.
\end{Remark}

\section{Existence of solution to $(S)$}
Basically we {prove} Theorem \ref{mt1} in this section. Before proving this, we give a Lemma that will be recalled in our work several times as {the} weak comparison principle.

\begin{Lemma}\label{weak_comp_principle}
Assume $\la>0$ and $u,v \in X_0(\Om)$ are weak solutions of
\begin{align}
&{A_\la}u= g_1 \; \text{in}\; \Om,\label{wcp-1}\\
&{A_\la}v = g_2 \; \text{in}\; \Om \label{wcp-2}
\end{align}
with $g_1,g_2 \in L^2(\Om)$ such that $g_1 \leq\; g_2$, where ${A_\la} :X_0(\Om)\cap \mc C \to (X_0(\Om))^*$ {(dual space of $X_0(\Om)$)} is defined as ${A_\la}(u):= u+\la((-\De)^su -u^{-q})$, with $\la>0$ fixed. Then $u\leq v$ a.e. in $\Om$. Moreover, for $g \in L^\infty(\Om)$ the problem
\begin{equation}\label{uniq}
A_\la u = g \; \text{in} \;\; \Om, \; u =0 \; \text{in}\;\; \mb R^n\setminus \Om
\end{equation}
has a unique solution in $X_0(\Om)$.
\end{Lemma}
\begin{proof}
Let $w=(u-v)$, then $w=w^+-w^-$ where $w^+=\max\{w,0\}$ and $w^-=\max\{-w,0\}$. Let $\Om^+ := \{x\in \Om: u(x)> v(x)\}$ and $\Om^- := \Om \setminus \Om^+$, then $\Om= \Om^+\cup \Om^-$. Multiplying \eqref{wcp-1} and \eqref{wcp-2} by $w^+$, integrating over $\mb R^n$ on both sides and subtracting, we get
{
\begin{align*}
&\int_{\Om^+} (u-v)^2dx+ \la \left( C_n^s\int_Q\frac{((u-v)(x)-(u-v)(y))(w^+(x)- w^+(y))}{|x-y|^{n+2s}}dxdy\right.\\
  &\quad \quad - \left.\int_{\Om^+}\left(\frac{1}{v^q}-\frac{1}{u^q}\right)(u-v)dx\right)
= \int_{\Om^+} (g_1-g_2)w^+ dx.
\end{align*}}
{Since for $(x,y)\in \Om \times \mc C\Om$, $((u-v)(x)-(u-v)(y))(w^+(x)-w^+(y))= (u-v)(x)w^+(x)\geq 0$  and for $(x,y)\in \Om^+ \times \Om^-$, $((u-v)(x)-(u-v)(y))w^+(x)\geq 0$ we get
\begin{equation}\label{wcp3n}
\begin{split}
 \int_{\Om^+} (u-v)^2dx &+ \la \left(C_n^s \int_{\Om^+}\int_{\Om^+} \frac{((u-v)(x)-(u-v)(y))^2}{|x-y|^{n+2s}}dxdy\right.\\
  &\quad \left.- \int_{\Om^+}\left(\frac{1}{v^q}-\frac{1}{u^q}\right)(u-v)dx\right)
\leq \int_{\Om^+} (g_1-g_2)w^+ dx.
 \end{split}
\end{equation}}
{We can also prove that $A_\la$} is a strictly monotone operator (for definition refer \cite{Barbu}). So left hand side of \eqref{wcp3n} is positive whereas  $\int_{\Om^+}(g_1-g_2)w^+ dx \leq 0$. Therefore we arrive at a contradiction which implies $u\leq v$ a.e. in $\Om$.
Then the uniqueness of \eqref{uniq} follows directly.
\QED
\end{proof}

\noi \textbf{Proof of Theorem \ref{mt1}:} For $\e >0$, we consider the following approximated problem corresponding to $(S)$ as
\begin{equation*}
(S_\e)\left\{
\begin{split}
&u+ \la \left( (-\De)^su-(u+\e)^{-q}\right)=g,\;\;  \; u>0 \; \text{in} \; \Om,\\
& u=0 \; \text{in}\; \mb R^n \setminus \Om.
\end{split}
\right.
\end{equation*}
Let $X_0^+(\Om)= \{u\in X_0(\Om): \; u\geq 0\}$. The energy functional associated to $(S_\e): X_0^+(\Om) \to \mb R$ is
\[E_\la(u) = \frac12 \int_\Om u^2 ~dx + \frac{\la}{2}\|u\|_{X_0(\Om)}^2 -\frac{\la}{1-q}\int_{\Om}({u}+\e)^{1-q}~dx- \int_{\Om}gu ~dx\]
which can shown to be weakly lower semicontinuous, coercive and strictly convex in $X_0^+(\Om)$. Since $X_0(\Om)$ is reflexive and $X_0^+(\Om)$ being a closed convex subset of $X_0(\Om)$, $E_\la$ has a unique global minimizer $u_{\la,\e} \in X_0^+(\Om)$ i.e. $u_{\la,\e} \geq 0$ a.e. in $\Om$. Let $\phi_{1,s}$ denotes the normalized first eigenfunction associated with first eigenvalue $\la_{1,s}$ of $(-\De)^s$  with Dirichlet boundary condition in $\mb R^n \setminus \Om$ i.e.
\[(-\De)^s\phi_{1,s}= \la_{1,s}\phi_{1,s} \; \text{in} \; \Om, \; \; \phi_{1,s}=0 \; \text{in}\; \mb R^n\setminus \Om,\]
where $0< \phi_{1,s} \in X_0(\Om)\cap L^{\infty}(\Om)$ is normalized by $\|\phi_{1,s}\|_{L^2(\Om)}=1$, refer {[\cite{varmethod}, Proposition $9$, p. $8$]}.
Also there exists $l>0$ such that $l\delta^s(x) \leq \phi_{1,s}(x)$ for a.e. $x \in \Om$ (see \cite{Roston-serra}). Since $g \in L^\infty(\Om)$, if we choose $m>0$ (depending on $\la,q$ and $g$) small enough such that (in the weak sense)
\[m \|\phi_{1,s}\|_\infty+ \la \la_{1,s}m\|\phi_{1,s}\|_{\infty}- \frac{\la}{m^q \|\phi_{1,s}+\e\|^q_\infty}<g,\]
then $m\phi_{1,s}$ forms a strict subsolution of $(S_\e)$ (independent of $\e$) i.e.
\begin{equation}\label{sp_mt1_1}
\left\{
\begin{split}
&m\phi_{1,s}+\la \left( (-\De)^s(m\phi_{1,s})- \frac{1}{(m\phi_{1,s}+\e)^q}\right)<g \; \text{in}\; \Om,\\
&m\phi_{1,s}=0 \; \text{in}\; \mb R^n\setminus \Om.
\end{split}
\right.
\end{equation}
 We define $w_{\e}:= (m\phi_{1,s}-u_{\la,\e})^+$ with the assumption that $\text{supp}(w_{\e})$ has non zero measure and for $t>0$, $\zeta(t):= E_\la(u_{\la,\e} +tw_{\e})$, then
\begin{align*}
\zeta^\prime(t)&= \int_{\Om}(u_{\la,\e} +tw_{\e})w_{\e} + \la C_n^s\int_Q\frac{( (u_{\la,\e} +t w_{\e})(x)-(u_{\la,\e} +tw_{\e} )(y))(w_{\e}(x)-w_{\e}(y))}{|x-y|^{n+2s}}~dxdy \\
&\quad -\la\int_{\Om}\frac{w_{\e}}{(u_{\la,\e}+tw_{\e}+\e)^q}-\int_{\Om}gw_{\e}
\end{align*}
in $(0,1]$.
 Since $u_{\la,\e}$ is the minimizer of $E_\la$, $\lim\limits_{t \to 0^+}{\zeta^\prime}(t)\geq 0$. Moreover, convexity of $E_\la$ assures that the map $t \mapsto \zeta^\prime(t)$ is {non decreasing}. {This implies $0\leq \zeta^\prime(0^+) \leq \zeta^\prime(1)$. Let us recall the following inequality for any $\psi$ being a convex Lipschitz function:
 \[(-\De)^s\psi(u) \leq \psi^\prime (u)(-\De)^su.\]
 Therefore using this with $\psi(x)=\max \left\{x,0\right\}$ and \eqref{sp_mt1_1}, we get
$\zeta^\prime(1)\leq \langle E_\la^\prime(m\phi_{1,s}), w_\e \rangle <0$ which is a contradiction. Hence $\text{supp}(w_\e)$ must have measure zero which implies
\begin{equation}\label{psp1}
m\phi_{1,s}\leq u_{\la,\e}.
\end{equation}
Using \eqref{psp1}, we can show that $E_\la$ is G{\^{a}}teaux differentiable in $u_{\la,\e}$ and as a result $u_{\la,\e}$ satisfies in the sense of distributions
\begin{align*}
u_{\la,\e} +\la(-\De)^s u_{\la,\e}= \la u_{\la,\e}^{-q} +g \; \text{in}\; \Om.
\end{align*}
Using Proposition $2.9$ of \cite{Silvestre}, we get $u_{\la,\e} \in C^{1,\alpha}(\mb R^n)$ for any $\alpha < 2\sigma -1$ where $2\sigma >1$.  Also since $g \in L^\infty(\Om)$, using Proposition $1.1$ (p. $277$) of \cite{Roston-serra} we get $u_{\la,\e} \in C^s(\mb R^n)$.
 Now we claim that $u_{\la,\e}$ is monotone increasing as $\e \downarrow 0^+$. Let $0<\e_1<\e_2$, then we show that $u_{\la,\e_1}> u_{\la,\e_2}$ in $\Om$. If possible, let $x_0\in \Om$ be such that $x_0 := \text{arg} \;\min\limits_{\overline\Om} (u_{\la,\e_1}-u_{\la,\e_2})$ and $u_{\la,\e_1}(x_0)\leq u_{\la,\e_2}(x_0)$. Then
\[(u_{\la,\e_1}-u_{\la,\e_2})+ \la (-\De)^s(u_{\la,\e_1}-u_{\la,\e_2})= \la \left( \frac{1}{(u_{\la,\e_1}+\e_1)^q}- \frac{1}{(u_{\la,\e_2}+\e_2)^q}\right)\]
which implies
\begin{align}
&(u_{\la,\e_1}-u_{\la,\e_2})(x_0)+ \la C_n^s \int_{\mb R^n}\frac{(u_{\la,\e_1}-u_{\la,\e_2})(x_0)-(u_{\la,\e_1}-u_{\la,\e_2})(y)}{|x_0-y|^{n+2s}}~dy \label{q>1eqn1}\\
&= \la \left( \frac{1}{(u_{\la,\e_1}(x_0)+\e_1)^q}- \frac{1}{(u_{\la,\e_2}(x_0)+\e_2)^q}\right).\label{q>1eqn3}
\end{align}
But we can see that  \eqref{q>1eqn1} is negative whereas \eqref{q>1eqn3} is positive which gives a contradiction. Therefore $x_0 \in \partial \Om$ and $u_{\la,\e_1}> u_{\la,\e_2}$ in $\Om$. Thus we get that $u_\la := \lim\limits_{\e \downarrow 0^+} u_{\la,\e} \geq m \phi_{1,s}$.
Let $w \in X_0^+(\Om)$ solves the problem
\begin{equation}\label{sp_mt1_4}
(-\De)^s w = w^{-q} \; \text{in}\; \Om,\;\;
w =0\; \text{in}\; \mb R^n\setminus \Om.
\end{equation}
Then from the proof of  Theorem $1.1$  of \cite{aJs}, we know that $w$ satisfies
\begin{align}
&k_1\phi_{1,s} \ln^{\frac12}\left(\frac{2}{\phi_{1,s}}\right) \leq w \leq k_2 \phi_{1,s} \ln^{\frac12}\left(\frac{2}{\phi_{1,s}}\right),\; \text{if}\; q=1\label{sp_mt1_5}\\
&k_1 \phi_{1,s}^{\frac{2}{q+1}}\leq w \leq k_2 \phi_{1,s}^{\frac{2}{q+1}},\; \text{if}\; q>1 \label{sp_mt1_6}
\end{align}
where $k_1,k_2>0$ are appropriate constants. Let $\overline{u}:= M_1w \in \mc C \cap C_0(\overline{\Om})$ for $M_1>0$. Then we can choose $M_1>>1$ (independent of $\e$) large enough such that
  \begin{align*}
  \overline{u}+\la \left( (-\De)^s\overline{u} - \frac{1}{(\overline{u} +\e)^q}\right)&= M_1w + \la\left(\frac{M_1}{w^q}- \frac{1}{(M_1w+\e)^q} \right)\\
  &\geq M_1w + \la\left(\frac{1}{(M_1w)^q}- \frac{1}{(M_1w+\e)^q} \right)>g \; \text{in}\; \Om.
  \end{align*}
  Using Lemma \ref{weak_comp_principle}, we get $u_{\la,\e} \leq \overline{u}$ which implies $u_\la \leq  \overline{u} = M_1w$. Now since $m \phi_{1,s}\leq u_\la\leq M_1w$ and both $w,\phi_{1,s} = 0$ in $\mb R^n \setminus \Om$, we get $u_\la = 0$ in $\mb R^n \setminus \Om$. Also $u_\la$ solves $(S)$ in the sense of distributions. Let $\underline{u}:= M_2w\in \mc C \cap C_0(\overline{\Om})$ then $M_2>0$ can be chosen small enough so that
  \begin{align*}
  M_2^{q+1}\left( 1+ \frac{w^{q+1}}{\la}\right) &\leq 1+ \frac{g(M_2w)^q}{\la} \;\text{in}\; \Om \\
  \text{i.e.}\; \underline u+\la (-\De)^s\underline{u} &< \frac{\la}{\underline u^q}+g \;\text{in}\; \Om.
  \end{align*}
  This implies that $\underline{u}$ forms a subsolution to $(S)$. We claim that $\underline{u}\leq u_\la$ in $\Om$. If possible, let $x_0\in \Om$ be such that $x_0 := \text{arg} \;\min\limits_{\overline\Om} (u_\la-\underline{u})$ and $u_\la(x_0)\leq \underline{u}(x_0)$. Then using the fact that $u_\la$ is a solution of $(S)$ in the sense of distributions and $\underline{u}$ is a subsolution of $(S)$, we get
 \begin{equation}\label{q>1eqn2}
 \begin{split}
(u_\la-\underline{u})(x_0) &+ \la \int_{\Om}\frac{(u_\la-\underline{u})(x_0)-(u_\la-\underline{u})(y)}{|x_0-y|^{n+2s}}dy\\
&\geq (u_\la-\underline{u})(x_0) + \la(-\De)^s(u_\la-\underline{u})(x_0) \geq \la \left( \frac{1}{u_\la^q(x_0)}- \frac{1}{\underline{u}^q(x_0)}\right).
\end{split}
\end{equation}
 This gives a contradiction, since left hand side of \eqref{q>1eqn2} is negative whereas right hand side of \eqref{q>1eqn2} is positive. Therefore we obtain
\[\underline{u} \leq u_\la \leq \overline{u} \]
which implies that $u_\la \in \mc C$, using \eqref{sp_mt1_5} and \eqref{sp_mt1_6}.
We now show that $u_\la \in X_0(\Om)$ and is a weak solution to $(S)$. Since $q(2s-1)<(2s+1)$, using the behavior of $u_\la$ with respect to $\delta$ function we get that $\displaystyle \int_{\Om}u_\la^{1-q}~dx < +\infty$. Also $\displaystyle \int_{\Om}u_\la^{-q}\phi~dx < +\infty$ for any $\phi \in X_0(\Om)$ from Hardy's inequality. Therefore using $\overline{C_c^\infty}^{\|\cdot\|_{X_0(\Om)}}= X_0(\Om)$ and Lebsegue dominated convergence theorem, we get that for any $\phi \in X_0(\Om)$
\[\int_{\Om} u_\la\phi + \la C_n^s \int_Q\frac{(u_\la(x)- u_\la(y))(\phi(x)-\phi(y))}{|x-y|^{n+2s}}~dxdy  - \int_{\Om}\left( \frac{\la}{u^q_\la}+ g\right)\phi~ dx=0.\]
That is $u_\la \in X_0(\Om) \cap \mc C$ is a weak solution to $(S)$. By Lemma \ref{weak_comp_principle}, uniqueness of $u_\la$ follows. Following the proof of Theorem $1.2$ in \cite{aJs}, we get that $u \in C^\alpha(\mb R^n)$ where $\alpha = s$ if $q<1$, $\alpha=s-\e$ if $q=1$, for any $\e>0$ small enough and $\alpha = \displaystyle\frac{2s}{q+1}$ if $q>1$. This completes the proof.}
\QED
\noi To prove the next result, we follow Lemma $3.6$ and Theorem $3.7$ of \cite{bbmp}.\\

\noi \textbf{Proof of Theorem \ref{geq-case}:} Consider the following approximated problem
\begin{equation*}
(P_{\la,k})\left\{
\begin{split}
u_k + \la \left((-\De)^su -\frac{1}{\left(u+\frac{1}{k}\right)^q} \right) &= g \; \text{in}\; \Om,\\
u_k &=0 \; \text{in}\; \mb R^n \setminus \Om.
\end{split}
\right.
\end{equation*}
By minimization argument we know that the solution $u_k$ to the problem $(P_{\la,k})$ belongs to $X_0(\Om)$. By weak comparison principle we get $u_k \leq u_{k+1}$ for all $k$. From the proof of Theorem \ref{mt1} we know that $ m\phi_{1,s}$ and $\overline u =M_1w$ forms subsolution and  supersolution  of $(P_{\la,k})$ respectively independent of $k$, where $w$ solves \eqref{sp_mt1_4} and $m$ is a sufficiently small whereas $M_1$ is a sufficiently large positive constant. Therefore
\begin{equation}\label{geq-case1}
0 \leq m\phi_{1,s} \leq  u_k \leq u_{k+1} \leq \overline{u},\; \text{for all}\; k.
\end{equation}
Since $g \in L^\infty(\Om)$ so Proposition $1.1$ of \cite{Roston-serra} gives that $u_k \in L^\infty(\Om)\cap C^s(\mb R^n)$ for all $k$. Therefore if $\tilde \Om\Subset \Om$ then there exists a constant $c_{\tilde \Om}>0$ such that
\begin{equation}\label{geq-case2}
u_k \geq c_{\tilde \Om} >0 \; \text{in}\; \tilde \Om.
\end{equation}
Let $u_\la := \lim\limits_{k \to \infty}u_k$. Then $u_\la$ solves $(S)$ in the sense of distributions. From the proof of Theorem \ref{mt1} we also know that for sufficiently small $M_2>0$, $\underline u = M_2w$ satisfies
 \[\underline u+ \la((-\De)^s \underline u- \underline u^{-q}) < g\; \text{in}\; \Om. \]
 Then following the arguments in proof of Theorem \ref{mt1} (refer \eqref{q>1eqn2}) we can show that $
 \underline u \leq u_\la \leq \overline u$ which implies that $u_\la \sim d^{\frac{2s}{q+1}}(x)$.
  Now for $ b> 1$ and $\beta \geq 1$, consider the function $\phi_\beta : [0,+\infty) \to [0,+\infty)$ defined as
\begin{equation*}
\phi_\beta(r)= \left\{
\begin{split}
r^\beta, \; \text{if}\; 0\leq r<b,\\
\beta b^{\beta-1}r-(\beta-1)b^\beta, \;  \text{if}\;  r\geq b >1.
\end{split}
\right.
\end{equation*}
Then $\phi_\beta$ is a lipschitz function with lipschitz constant $\beta b^{\beta-1}$. We have $q>1$. So let
\begin{equation}\label{geq-case8}
\beta > \max\left\{1,  \left(1-\frac{1}{2s} \right)\left( \frac{q+1}{2}\right) \right\} \geq 1.
 \end{equation}
 Then if $(2\beta - 1-q)<0$ then from $u_\la \sim d^{\frac{2s}{q+1}}(x)$ and \eqref{geq-case8} we get
\begin{equation}\label{geq-case7}
\int_{\Om} \frac{\phi_\beta^\prime(u_\la) \phi_\beta(u_\la)}{u_\la^q}~dx < +\infty.
\end{equation}
Since $\phi^\prime_\beta(u)\phi_\beta(u)\leq \beta u^{2\beta-1}$ so using \eqref{geq-case7}, $u_k\uparrow u_\la$ as $k \to \infty$ and monotone convergence theorem we get that
\begin{equation}\label{geq-case4}
\int_{\Om} \frac{\phi_\beta^\prime(u_k) \phi_\beta(u_k)}{u_k^q}~dx < +\infty \; (\text{independent of } k).
\end{equation}
Also \eqref{geq-case4} holds true when $(2\beta-1-q)\geq 0$ which follows from the uniform bound of $\{u_k\}$ in $L^\infty (\Om)$. Since it holds
\[(-\De)^s \phi_\beta (u_k) \leq \phi_\beta^\prime (u_k) (-\De)^s u_k,\]
therefore using \eqref{geq-case4} we get
\begin{equation*}
\begin{split}
\int_{\mb R^n} \phi_\beta(u_k) (-\De)^s\phi_\beta(u_k) & \leq \frac{1}{\la}\int_{\Om} (g-u_k)\phi_\beta^\prime(u_k) \phi_\beta(u_k)~dx + \int_{\Om} \frac{\phi_\beta^\prime(u_k) \phi_\beta(u_k)}{u_k^q}~dx\\
 & \leq  \beta \left( \frac{\|g\|_\infty \|\overline u\|_{L^{2\beta-1}(\Om)}}{\la}+ C \right),
\end{split}
\end{equation*}
where $C>0$ is a constant independent of $k$. Passing on the limit as $b \to \infty$ we get $\{u_k^\beta\}$ is uniformly bounded in $X_0(\Om)$. By weak lower semicontinuity of norms we have
\[ \|u_\la^{\beta}\| \leq \liminf_{k \to \infty} \|u_k^{\beta}\| < +\infty  \]
which implies $u_\la^{\beta} \in X_0(\Om)$. Thus $u_\la^{\beta} \in L^{2^*_s}(\Om)$ and since $\beta2^*_s >1$ we get $u_\la \in L^1(\Om)$. Now let $\psi \in \Theta$ such that $\text{supp}(\psi)= \tilde \Om \Subset \Om$ then by Lebesgue dominated convergence theorem we get
\[\lim_{k \to \infty} \int_{\mb R^n} u_k (-\De)^s \psi~dx = \int_{\mb R^n} u_\la (-\De)^s \psi~dx < +\infty. \]
Using \eqref{geq-case2} we get
\[0 \leq \left| \left( \frac{g-u_k}{\la} + \frac{1}{\left(u_k + \frac{1}{k}\right)^q} \right)\psi \right| \leq \left( \frac{|g| + |\overline{u}|}{\la}+ \frac{1}{c_{\tilde \Om}^q} \right)|\psi| \in L^1(\Om).\]
Therefore using Lebesgue dominated convergence theorem again we obtain
\begin{equation}\label{geq-case5}
\int_{\mb R^n}u_\la (-\De)^s \psi = \lim_{k \to \infty} \int_\Om \left(\frac{g-u_k}{\la} + \frac{1}{\left(u_k + \frac{1}{k}\right)^q}\right)\psi~dx= \int_\Om \left(\frac{g-u_\la}{\la} + \frac{1}{u_\la^q}\right)\psi~dx.
\end{equation}
We claim that $u_\la \notin X_0(\Om)$. On contrary if $u_\la \in X_0(\Om)$ then using Lemma $3.1$ of \cite{TJS} and monotone convergence theorem, we can easily show that \eqref{geq-case5} holds for any $\psi \in X_0(\Om)$. Therefore $u_\la \in X_0(\Om)$ solves $(S)$ in the weak sense  and
we get
\[\frac{1}{u^q_\la}= \frac{1}{\la}(u_\la-g)+ (-\De)^su_\la \in (X_0(\Om))^*.  \]
Using \eqref{geq-case1} this implies that
\[\int_\Om \overline{u}^{1-q}~dx \leq \int_\Om u^{1-q}_\la~dx <+\infty \]
which contradicts the definition of $\overline u$.  \QED

Now following the proof of Lemma $6.1$ of \cite{TS1}, we can show that \eqref{geq-case5} holds for any $\psi \in X_0(\Om)$.

\section{Existence of solution to $(G^s_t)$ and its regularity}
We prove Theorem \ref{mt2} and Proposition \ref{para_sing_prop1} in this section. We use the method of semi-discretization in time along with implicit Euler method to prove Theorem \ref{mt2}.
\begin{Theorem}\label{mt2-prt1}
If $h(t,x) \in L^{\infty}(\Lambda_T)$, $u_0\in X_0(\Om)\cap \mc C$, $q>0$ and $q(2s-1)<(2s+1)$, then there exists a unique weak solution $u \in \mathcal{A}_{\Lambda_T}\cap \mathcal C$ of the problem $(G^s_t)$.
\end{Theorem}
\begin{proof}
 Let $\De_t = \frac{T}{n}$ and for $0\leq k \leq n$, define $t_k:= k\De_t $. Also define
\[h_k(x):= \frac{1}{\De_t}\int_{t_{k-1}}^{t_k} h(\tau,x)d\tau \; \text{for}\; x \in \Om.\]
Since $h \in L^\infty(\Lambda_T)$, we get $h_k \in L^\infty(\Om)$ and $\|h_k\|_{\infty} \leq \|h\|_{L^\infty(\Lambda_T)}$. Then we define
\[h_{\De_t}(t,x):= h^k(x), \; \text{when}\; t \in [t_{k-1}, t_k),\; 1\leq k \leq n\]
and get that $h_{\De_t} \in L^\infty(\Lambda_T)$. For $1<p<+\infty$,
\begin{equation}\label{mt2_4}
\begin{split}
\|h_{\De_t}\|_{L^p(\Lambda_T)}\leq (|\Om|T)^{\frac{1}{p}}\|h\|_{L^\infty(\Lambda_T)}
\end{split}
\end{equation}
and $h_{\De_t} \to h$ in $L^p(\Lambda_T)$ as $\De_t \to 0$. Taking $\la = \De_t$ and $g= \De_t h_k + u^{k-1} \in L^\infty(\Om)$ in $(S)$, using Theorem \ref{mt1} we define the sequence $\{u^k\} \subset X_0(\Om)\cap \mc C$ as solutions to problem
\begin{equation}\label{mt2_1}
\left\{
\begin{split}
\frac{u^k-u^{k-1}}{\De_t} +(-\De)^su^k - \frac{1}{(u^k)^q}&=h_k \; \text{in}\; \Om,\\
u^k&=0\; \text{in}\; \mb R^n\setminus \Om,
\end{split}
\right.
\end{equation}
where $u^0 = u_0 \in X_0(\Om)\cap \mc C$. Now, for $1\leq k\leq n$, we define
\begin{equation}\label{mt2_2}
\forall t \in [t_{k-1}, t_k),\left\{
\begin{split}
u_{\De_t}(t,x)&:= u^k(x)\\
\tilde{u}_{\De_t}(t,x) &:= \frac{(u^k(x)-u^{k-1}(x))}{\De_t}(t-t_{k-1})+u^{k-1}(x).
\end{split}
\right.
\end{equation}
Then $u_{\De_t}$ and $\tilde{u}_{\De_t}$ satisfies
\begin{equation}\label{mt2_3}
\frac{\partial \tilde{u}_{\De_t}}{\partial t} +(-\De)^su_{\De_t}- \frac{1}{u_{\De_t}^q}= h_{\De_t} \in L^{\infty}(\Lambda_T).
\end{equation}
At first, we establish some a priori estimates for $u_{\De_t}$ and $\tilde{u}_{\De_t}$ independent of $\De_t$. Multiplying \eqref{mt2_1} by $\De_t u^k$, integrating over $\mb R^n$ and summing from $k=1$ to $n^\prime \leq n$, using Young's inequality and \eqref{mt2_4} we get for a constant $C>0$
\begin{equation}\label{mt2_5}
\begin{split}
&\sum_{k=1}^{n^\prime} \int_\Om(u^k-u^{k-1})u^k dx + \De_t\sum_{k=1}^{n^\prime}\left( \|u^k\|_{X_0(\Om)}^2 -\int_\Om (u^k)^{1-q}dx \right)= \De_t \sum_{k=1}^{n^\prime}\int_{\Om}h^k u^k dx\\
 &\leq \De_t \sum_{k=1}^{n^\prime}\int_\Om \frac{|h^k|^2}{2}dx + \De_t \sum_{k=1}^{n^\prime} \int_\Om\frac{|u^k|^2}{2}dx \leq \frac{T}{2}|\Om|\|h\|_{L^\infty(\Lambda_T)}^2 + \frac{C\De_t}{2} \sum_{k=1}^{n^\prime}\|u^k\|_{X_0(\Om)}^2.
\end{split}
\end{equation}
As inequality $(2.7)$ of Theorem $0.9$ in \cite{bbg}, we can estimate the first term of \eqref{mt2_5} as
\begin{equation}\label{mt2_6}
\sum_{k=1}^{n^\prime} \int_\Om (u^k-u^{k-1})u^k dx = \frac12 \sum_{k=1}^{n^\prime}\int_\Om  |u^k-u^{k-1}|^2 dx +\frac12 \int_\Om |u^{n^\prime}|^2 dx - \frac12 \int_\Om |u_0|^2dx.
\end{equation}
Let $v$ and $w$ solves \eqref{sp_mt1_4} and define
\begin{equation*}
\underline{u}= m w \; \text{and}\; \overline{u}=M w
\end{equation*}
 where {$m >0$ is small enough  and $M>0$ is large enough chosen in such a way that}
\begin{equation*}
\left\{
\begin{split}
(-\De)^s\underline{u}-\frac{1}{\underline{u}^q} &\leq -\|h\|_{L^\infty(\Lambda_T)} \; \text{in}\; \Om,\\
(-\De)^s\overline{u}-\frac{1}{\overline{u}^q} &{\geq \|h\|_{L^\infty(\Lambda_T)}} \; \text{in}\; \Om.
\end{split}
\right.
\end{equation*}
Since $u_0 \in \mc C$, we can always choose $\underline{u}$ and $\overline{u}$ such that it satisfies the above inequalities and $\underline{u} \leq u_0 \leq \overline{u} $. Applying Lemma \ref{weak_comp_principle} iteratively we get $\underline{u} \leq u^k \leq \overline{u}$ for all $k$. This implies for a.e. $(t,x) \in [0,T]\times \Om$,
\begin{equation}\label{mt2_7}
\underline{u}(x) \leq u_{\De_t}(t,x), \tilde{u}_{\De_t}(t,x)\leq \overline{u}(x)
\end{equation}
i.e. $u_{\De_t},\tilde{u}_{\De_t} \in \mc C$ uniformly. Now since $q(2s-1)<(2s+1)$ we can estimate the singular term in \eqref{mt2_5} as
\begin{equation}\label{mt2_8}
\De_t \sum_{k=1}^{n^\prime}\int_\Om (u^k)^{1-q}dx \leq\;\left\{
\begin{split}
T \int_\Om\overline{u}^{1-q}dx < +\infty\; \text{if} \; q \leq 1, \\
T \int_\Om\underline{u}^{1-q}dx < +\infty\; \text{if} \; q > 1.
\end{split}
\right.
\end{equation}
Since $u^k \in L^\infty(\Om)$ for all $k$, by the definition of $u_{\De_t}$ and $\tilde{u}_{\De_t}$ we easily get that \begin{equation}\label{mt2_13}
u_{\De_t},\tilde{u}_{\De_t} \; \text{is bounded in}\; L^\infty([0,T], L^\infty(\Om)).
\end{equation}
 We see that for $t \in [t_{k-1},t_k)$
\[\|\tilde{u}_{\De_t}(t,\cdot)\|_{X_0(\Om)}= \left\|\frac{(t-t_{k-1})}{\De_t}u^k+ \frac{(\De_t- t+t_{k-1})}{\De_t}u^{k-1}\right\|_{X_0(\Om)} \leq \|u^k\|_{X_0(\Om)}+ \|u^{k-1}\|_{X_0(\Om)}.\]
Integrating both sides of \eqref{mt2_5} over $(t_{k-1},t_k)$ and using \eqref{mt2_6}, \eqref{mt2_8} we get that \begin{equation}\label{mt2_14}
u_{\De_t},\tilde{u}_{\De_t}  \; \text{is bounded in}\; L^2([0,T], X_0(\Om)).
\end{equation}
 We now try to obtain a second energy estimate.
 Multiplying \eqref{mt2_1} by $u^k-u^{k-1}$, integrating over $\mb R^n$ and summing from $k=1$ to $n^\prime \leq n$, using Young's inequality and \eqref{mt2_4} we get
\begin{equation}\label{secnd-ener}
\begin{split}
&\De_t\sum_{k=1}^{n^\prime} \int_\Om\left(\frac{u^k-u^{k-1}}{\De_t}\right)^2 dx + {\sum_{k=1}^{n^\prime} \int_{\mb R^n} ((-\De)^su^k(x))(u^k-u^{k-1})(x)dx} -\sum_{k=1}^{n^\prime} \int_\Om \frac{(u^k-u^{k-1})}{(u^k)^{q}}dx \\
&= \De_t\sum_{k=1}^{n^\prime}\int_{\Om}\frac{h^k (u^k-u^{k-1})}{\De_t}~ dx \leq \frac{\De_t}{2} \sum_{k=1}^{n^\prime}\left(\int_\Om{|h^k|^2}dx + \int_\Om\left(\frac{u^k-u^{k-1}}{\De_t}\right)^2dx\right)\\
%& \leq \frac{T}{2}|\Om|\|h\|_{L^\infty(\Lambda_T)}^2 + \frac{C\De_t}{2} \sum_{k=1}^{n^\prime}\|u^k\|_{X_0(\Om)}^2
\end{split}
\end{equation}
which implies
\begin{equation}\label{mt2_9}
\begin{split}
&\frac{\De_t}{2}\sum_{k=1}^{n^\prime} \int_\Om\left(\frac{u^k-u^{k-1}}{\De_t}\right)^2 dx + \sum_{k=1}^{n^\prime} \int_{\mb R^n}((-\De)^su^k(x))(u^k-u^{k-1})(x)dx\\
& \quad \quad-\sum_{k=1}^{n^\prime} \int_\Om \frac{(u^k-u^{k-1})}{(u^k)^{q}}dx \leq \frac{|\Om|T}{2}\|h\|_{L^{\infty}(\Lambda_T)}^2.
\end{split}
\end{equation}
By convexity of the term $\frac{-1}{1-q}\int_\Om u^{1-q}dx $, we have
\begin{equation}\label{convexity1}
 \frac{1}{1-q} \int_\Om \left( (u^{k-1})^{1-q}-(u^k)^{1-q} \right)dx \leq -\int_\Om \frac{u^k-u^{k-1}}{(u^k)^q}~dx.
 \end{equation}
Also
\begin{equation}\label{convexity}
{\frac12 \left(\|u^k\|_{X_0(\Om)}^2-\|u^{k-1}\|_{X_0(\Om)}^2\right) \leq  \int_{\mb R^n} ((-\De)^su^k(x))(u^k-u^{k-1})(x)dx}.
\end{equation}
Therefore \eqref{mt2_9} gives
\begin{equation}\label{mt2_10}
\begin{split}
\frac{\De_t}{2}\sum_{k=1}^{n^\prime}& \int_\Om\left(\frac{u^k-u^{k-1}}{\De_t}\right)^2 dx + \frac12 \left(\|u^{n^\prime}\|_{X_0(\Om)}^2-\|u_0\|_{X_0(\Om)}^2\right)\\
 &\quad + \frac{1}{1-q} \int_\Om \left( (u_0)^{1-q}-(u^{n^\prime})^{1-q} \right)dx \leq \frac{|\Om|T}{2}\|h\|_{L^{\infty}(\Lambda_T)}^2.
\end{split}
\end{equation}
Integrating over $(t_{k-1},t_k)$ on both sides of \eqref{mt2_10} and using \eqref{mt2_8}, we get
\[\frac{\De_t}{2} \int_{\Lambda_T}\left|\frac{\partial \tilde{u}_{\De_t}}{\partial t}\right|^2~dxdt < +\infty\]
which implies
\begin{equation}\label{mt2_11}
\frac{\partial \tilde{u}_{\De_t}}{\partial t} \; \text{is bounded in}\; L^2(\Lambda_T)\; \text{uniformly in}\; \De_t.
\end{equation}
Using definition of $u_{\De_t}$ and $\tilde {u}_{\De_t}$, we have that
\begin{equation}\label{mt2_12}
{u_{\De_t} \; \text{and}\; \tilde{u}_{\De_t} \; \text{are bounded in}\; L^\infty([0,T]; X_0(\Om))\; \text{uniformly in}\; \De_t.}
\end{equation}
Moreover, there exists a constant {$C>0$ (independent of $\De_t$)} such that
\begin{equation}\label{mt2_15}
\|u_{\De_t}-\tilde{u}_{\De_t}\|_{L^\infty([0,T];L^2(\Om))} \leq \max_{1\leq k\leq n}\|u^k - u^{k-1}\|_{L^2(\Om)} \leq C(\De_t)^{\frac12}.
\end{equation}
Using \eqref{mt2_13} and \eqref{mt2_12}, we get
\[u_{\De_t} \; \text{and}\; \tilde{u}_{\De_t} \; \text{are bounded in}\; L^\infty([0,T]; X_0(\Om)\cap L^\infty(\Om))\; \text{uniformly in}\; \De_t.\]
{ Using $\text{weak}^*$ and weak compactness results}, we say that as $\De_t \to 0^+$(i.e. $n \to \infty$), upto a subsequence
\begin{equation}\label{mt2_21}
\begin{split}
\tilde{u}_{\De_t} \xrightharpoonup{\text{*}} u  ,\;\;
u_{\De_t} \xrightharpoonup{\text{*}} v\; \text{in}\; L^\infty([0,T]; X_0(\Om)\cap L^\infty(\Om))\;\;
\text{and}\;\frac{\partial \tilde{u}_{\De_t}}{\partial t} \rightharpoonup \frac{\partial u}{\partial t} \; \text{in}\; L^2(\Lambda_T)
\end{split}
\end{equation}
where $u,v \in  L^\infty([0,T]; X_0(\Om)\cap L^\infty(\Om))$ such that $\frac{\partial u}{\partial t} \in L^2(\Lambda_T)$. From \eqref{mt2_15}, we confer that $u \equiv v$. Also from \eqref{mt2_7}, we get that $\underline{u} \leq u \leq \overline{u}$. Thus, $u \in \mc A(\Lambda_T)\cap \mc C$.

Now we will prove that $u$ is a weak solution to $(G^s_t)$. First we see that for a.e. $ x \in \Om$, $\tilde{u}_{\De_t}(\cdot,x) \in C([0,T])$. By \eqref{mt2_11}, we get that $\frac{\partial \tilde{u}_{\De_t}}{\partial t}$ is bounded in {$L^2(\Lambda_T)$ uniformly in $\De_t$}.
{Also, $\{\tilde{u}_{\De_t}\}$ is a bounded family in $X_0(\Om)$ and the embedding of $X_0(\Om)$ into $L^2(\Om)$ is compact. If we define
  \[W:= \left\{u \in C([0,T];X_0(\Om)): \; \frac{\partial u}{\partial t}\in L^2(\Lambda_T)\right\},\]
  then by Aubin-Lions-Simon Lemma, the embedding $W$ into $C([0,T];L^2(\Om))$ is compact. Therefore, we get that $\{\tilde{u}_{\De_t}\}$ is compact in $C([0,T];L^2(\Om))$.} Using $\underline{u} \leq \tilde{u}_{\De_t}\leq \overline{u}$ again, we get that $\{\tilde{u}_{\De_t}\}$ is compact in $C([0,T];L^p(\Om))$, $1<p<\infty$ and therefore as $\De_t \to 0^+$, upto a subsequence
\begin{equation}\label{mt2_16}
\tilde{u}_{\De_t} \to u \; \text{in}\; C([0,T];{L^2(\Om)}).
\end{equation}
This along with \eqref{mt2_15} gives that as $\De_t \to 0^+$
\begin{equation}\label{mt2_17}
u_{\De_t} \to u \; \text{in}\; L^\infty([0,T]; {L^2(\Om)}).
\end{equation}
Using $(u_{\De_t}-u)$ as the test function in \eqref{mt2_3},  we get
\[\int_0^T \int_{\mb R^n}\left(\frac{\partial \tilde{u}_{\De_t}}{\partial t}+(-\De)^s u_{\De_t} - u_{\De_t}^{-q}\right)(u_{\De_t}-u)dxdt= \int_{\Lambda_T}h_{\De_t}(u_{\De_t}-u)dxdt.\]
Also using \eqref{mt2_17}, we know that $\int_{\Lambda_T}\frac{\partial u}{\partial t}(\tilde{u}_{\De_t}-u)dxdt \to 0$ as $\De_t \to 0^+$. Hence
\begin{equation}\label{mt2_18}
\begin{split}
&\int_{\Lambda_T}\left(\frac{\partial \tilde{u}_{\De_t}}{\partial t}-\frac{\partial u}{\partial t}\right)(\tilde{u}_{\De_t}-u)dxdt-\int_{\Lambda_T} u_{\De_t}^{-q}(u_{\De_t}-u)dxdt\\
&\quad {+\int_0^T \langle(-\De)^s u_{\De_t},(u_{\De_t}-u)\rangle dt= \int_{\Lambda_T}h_{\De_t}(u_{\De_t}-u)dxdt+ o_{\De_t}(1).}
\end{split}
\end{equation}
{By \eqref{mt2_7}, we have $ u_{\De_t}^{-q} \leq \underline{u}^{-q}$. Also since $\underline{u}\leq u \leq \overline{u}$, we apply Lebesgue Dominated convergence theorem with \eqref{mt2_17} to get}
\begin{equation*}
\int_0^T \int_\Om u_{\De_t}^{-q}(u_{\De_t}-u)dxdt \leq \int_0^T\int_\Om \underline{u}^{-q}(u_{\De_t}-u)dxdt = o_{\De_t}(1).
\end{equation*}
{ Similarly using \eqref{mt2_4} and \eqref{mt2_17} along with Lebesgue theorem, we get}
\begin{equation*}
\int_{\Lambda_T}h_{\De_t}(u_{\De_t}-u)dxdt= o_{\De_t}(1).
\end{equation*}
Using integration by parts and the fact that $\tilde{u}_{\De_t}(0,x)= u(0,x)= u_0$ in a.e. $\Om$,
we get
\[2\int_{\Lambda_T}\left(\frac{\partial \tilde{u}_{\De_t}}{\partial t}-\frac{\partial u}{\partial t}\right)(\tilde{u}_{\De_t}-u)dxdt= \int_\Om (\tilde{u}_{\De_t}-u)^2(T)dt. \]
Therefore, \eqref{mt2_18} implies
\begin{equation*}
\frac12\int_\Om (\tilde{u}_{\De_t}-u)^2(T)dt {+ \int_0^T\langle(-\De)^s u_{\De_t}- (-\De)^s u,u_{\De_t}-u\rangle dt}= o_{\De_t}(1)
\end{equation*}
where we used the fact that $\int_0^T \langle (-\De)^su, u_{\De_t}-u\rangle dt = o_{\De_t}(1)$ which follows from \eqref{mt2_17}. Since $u \not\equiv 0$ identically in $\Lambda_T$, using \eqref{mt2_17} we get
\[\int_0^T \|(u_{\De_t}-u)(t,\cdot)\|_{X_0(\Om)}^2dt= o_{\De_t}(1).\]
{Let $(X_0(\Om))^*$ denotes the dual of $X_0(\Om)$. Then the above equations suggest that as $\Delta_t \to 0$}
\begin{equation}\label{mt2_19}
(-\De)^su_{\De_t} \to (-\De)^su \; \text{in}\; L^2([0,T];(X_0(\Om))^*).
\end{equation}
From \eqref{mt2_7}, for any $\phi \in X_0(\Om)$, using Hardy's inequality and $q(2s-1)<(2s+1)$ we have
\begin{equation*}
\int_\Om | \phi (u_{\De_t})^{-q} |dx \leq \int_{\Om}|\phi||\underline{u}^{-q}|dx \leq \left(\int_\Om\frac{1}{\delta^{2s(q-1)/(q+1)}(x)}dx \right)^{\frac12}\left(\int_\Om \frac{\phi^2}{\delta^{2s}(x)}dx \right)^{\frac12}<+\infty.
\end{equation*}
Therefore using Lebesgue Dominated convergence theorem we get
\begin{equation}\label{mt2_20}
\frac{1}{(u_{\De_t})^q} \to \frac{1}{u^q} \; \text{in}\; L^\infty([0,T];(X_0(\Om))^*) \; \text{as}\; \De_t \to 0^+.
\end{equation}
Finally, we get $u \in \mc A(\Lambda_T)$ and for any $\phi \in \mc A(\Lambda_T)$ passing on the limit $\De_t \to 0^+$ in
\begin{equation*}
\begin{split}
&\int_{\Lambda_T}\frac{\partial \tilde{u}_{\De_t}}{\partial t}\phi~dxdt+ \int_0^T \int_{\mb R^n} (-\De)^su_{\De_t}\phi ~dxdt-\int_{\Lambda_T} \frac{1}{u_{\De_t}^{q}}\phi ~dxdt= \int_{\Lambda_T}h_{\De_t}\phi~dxdt,
\end{split}
\end{equation*}
using \eqref{mt2_4}, \eqref{mt2_21}, \eqref{mt2_19} and \eqref{mt2_20}, we get
\begin{equation}\label{mt2_23}
\begin{split}
&\int_{\Lambda_T}\frac{\partial u}{\partial t}\phi~dxdt+ \int_0^T \int_{\mb R^n}(-\De)^su\phi ~dxdt-\int_{\Lambda_T} \frac{1}{{u}^{q}}\phi ~dxdt= \int_{\Lambda_T}h\phi~dxdt.
\end{split}
\end{equation}
That is, $u$ is a weak solution to $(G_t^s)$.

Now we show the uniqueness of $u$ as solution of $(G^s_t)$ such that $u(t,\cdot) \in \mc C$, for all $t \in [0,T]$. On contrary, let $v$ such that $v(t,\cdot) \in \mc C$, for all $t \in [0,T]$ distinct from $u$ be another weak solution to $(G^s_t)$. Then for any $t \in [0,T]$, we have
\begin{equation*}
\begin{split}
\int_\Om \frac{\partial (u-v)}{\partial t}(u-v)(t,x)~dx &+ \int_{\mb R^n}((-\De)^s(u-v))(u-v)(t,x)~dx\\
& \quad - \int_\Om \left( \frac{1}{u^{q}}- \frac{1}{v^q}\right)(u-v)dx=0
\end{split}
\end{equation*}
which implies
\begin{equation*}
\begin{split}
\frac{\partial}{\partial t}\left( \int_\Om \frac12 (u-v)^2(t,x)~dx\right)= -\|(u-v)(t,\cdot)\|_{X_0(\Om)}^2 + \int_\Om \left( \frac{1}{u^{q}}- \frac{1}{v^q}\right)(u-v)(t,x)dx\leq 0.
\end{split}
\end{equation*}
So we see that the function $E:[0,T]\to \mb R$ defined as
\[E(t):= \frac12 \int_\Om (u-v)^2(t,x)~dx\]
is a decreasing function. Then since $u,v$ are distinct, we get
$0 < E(t) \leq E(0) =0$
which implies $E(t)=0$, for all $t\in [0,T]$. Hence $u \equiv v$.\QED
\end{proof}

\begin{Theorem}\label{mt2-prt2}
The unique weak solution $u$ of $(G^s_t)$ (as obtained in Theorem \ref{mt2-prt1}) belongs to
$ C([0,T]; X_0(\Om))$ and ${u(t)} \in \mc C$ uniformly for each $t \in [0,T]$. Also, $u$ satisfies \eqref{mt2_eq}.
\end{Theorem}
\begin{proof}
We first show that $u \in C([0,T]; X_0(\Om))$ and then establish \eqref{mt2_eq} in order to complete the proof of this theorem. From \eqref{mt2_21}, we already have $u \in C([0,T];L^2(\Om))$ which implies that the map $\tilde{u}:[0,T]\to X_0(\Om)$ defined as $[\tilde{u}(t)](x):= u(t,x)$ is weakly continuous.  Also \eqref{mt2_16} gives $u \in L^\infty([0,T];X_0(\Om))$, which implies $\tilde{u}(t) \in X_0(\Om)$ and $\|\tilde{u}(t)\|_{X_0(\Om)}\leq \lim\inf\limits_{t\to t_0} \|\tilde{u}(t)\|_{X_0(\Om)}$ for all $t_0 \in [0,T]$. Multiplying \eqref{mt2_1} by $u^k-u^{k-1}$, integrating over $\mb R^n$ on both sides and summing from $k=n^{\prime\prime}$ to $n^{\prime}$ ($n^\prime$ has been considered in \eqref{secnd-ener}) and using \eqref{convexity}  we get
\begin{equation*}
\begin{split}
\frac{\De_t}{2}\sum_{k=n^{\prime \prime}}^{n^\prime}& \int_\Om\left(\frac{u^k-u^{k-1}}{\De_t}\right)^2 dx + \frac12 \left(\|u^{n^\prime}\|_{X_0(\Om)}^2-\|u^{n^{\prime \prime} -1}\|_{X_0(\Om)}^2\right)\\
 &\quad + \frac{1}{1-q} \int_\Om \left( \left(u^{n^{\prime \prime}-1}\right)^{1-q}-\left(u^{n^\prime}\right)^{1-q} \right)dx \leq \sum_{k=n^{\prime \prime}}^{n^\prime} \int_\Om h_{\De_t}(u^k-u^{k-1})dx.
\end{split}
\end{equation*}
For any $t_1 \in [t_0,T]$, we take $n^{\prime\prime}$ and $n^{\prime}$ such that $n^{\prime\prime}\De_t \to t_1$ and $n^\prime \De_t\to t_0$ as $\De_t\to 0^+$. Using \eqref{mt2_4}, \eqref{mt2_15}, \eqref{mt2_17} and \eqref{mt2_20}, from above inequality we get
\begin{equation}\label{mt2_22}
\begin{split}
& \int_{t_0}^{t_1} \int_\Om \left( \frac{\partial u}{\partial t}\right)^2 dxdt+ \frac{1}{2}\|u(t_1,\cdot)\|_{X_0(\Om)}^2-\frac{1}{1-q}\int_\Om u(t_1)^{1-q}dx\\
&\quad \leq \int_{t_0}^{t_1} \int_\Om h\frac{\partial u}{\partial t}~ dxdt+ \frac{1}{2}\|u(t_0,\cdot)\|_{X_0(\Om)}^2-\frac{1}{1-q}\int_\Om u(t_0)^{1-q}dx.
\end{split}
\end{equation}
Since $u \in L^\infty([0,T];L^p(\Om))$ for $1<p<\infty $, passing on the limit $t_1 \to t_0^+$, we get
\[\lim\sup\limits_{t_1\to t_0^+}\|u(t_1,\cdot)\|_{X_0(\Om)}\leq \|u(t_0,\cdot)\|_{X_0(\Om)}.\]
Therefore $\lim\limits_{t_\to t_0^+}\|u(t,\cdot)\|_{X_0(\Om)}= \|u(t_0,\cdot)\|_{X_0(\Om)}$ which implies that $u$ is right continuous on $[0,T]$. Let us now prove the left continuity and assume $t_1>t_0$. Let $0< r \leq t_1-t_0$. Define
\[\sigma_r(z):= \frac{u(z+r)-u(r)}{r}.\]
Since $u$ is a weak solution to $(G^s_t)$, taking $\sigma_r(u)$ as the test function in $(G^s_t)$, integrating over $(t_0,t_1)\times \mb R^n$ and using \eqref{convexity1} we get
\begin{equation*}
\begin{split}
&\int_{t_0}^{t_1}\int_\Om \frac{\partial u}{\partial t}\sigma_r(u)~dxdt+ \frac{1}{2r}\int_{t_0}^{t_1}\int_{\mb R^n}((-\De)^su(t+r,x)- (-\De)^s u(t,x))dxdt\\
&\quad - \frac{1}{r(1-q)}\int_{t_0}^{t_1}\int_\Om (u^{1-q}(t+r,x)-u^{1-q}(t,x))dxdt\geq \int_{t_0}^{t_1}\int_\Om\sigma_r(u)dxdt.
\end{split}
\end{equation*}
 Then it is an easy task to get
\begin{equation}\label{mt2_24}
\begin{split}
&\int_{t_0}^{t_1}\int_\Om \frac{\partial u}{\partial t}\sigma_r(u)~dxdt+ \frac{1}{2r}\left(\int_{t_1}^{t_1+r}\int_{\mb R^n}(-\De)^su(t,x)dxdt- \int_{t_0}^{t_0+r}\int_{\mb R^n}(-\De)^s u(t,x)dxdt\right)\\
&\quad - \frac{1}{r(1-q)}\left(\int_{t_1}^{t_1+r}\int_\Om u^{1-q}(t,x)dxdt-\int_{t_0}^{t_0+r}\int_\Om u^{1-q}(t,x)dxdt\right)\geq \int_{t_0}^{t_1}\int_\Om\sigma_r(u)dxdt{.}
\end{split}
\end{equation}
Since $u$ is right continuous in $X_0(\Om)$, using Lebesgue Dominated Convergence theorem we get the following as $r \to 0^+$:
\begin{equation*}
\begin{split}
\frac{1}{r} \int_{t_1}^{t_1+r}\int_{\mb R^n}(-\De)^su(t,x)dxdt &\to \int_{\mb R^n}(-\De)^su(t_1,x)dx,\\
\frac{1}{r} \int_{t_0}^{t_0+r}\int_{\mb R^n}(-\De)^su(t,x)dxdt &\to \int_{\mb R^n}(-\De)^su(t_0,x)dx,\\
\frac{1}{r} \int_{t_1}^{t_1+r}\int_\Om u^{1-q}(t,x)dxdt & \to \int_\Om u^{1-q}(t_1,x)dxdt,\\
\frac{1}{r} \int_{t_0}^{t_0+r}\int_\Om u^{1-q}(t,x)dxdt & \to \int_\Om u^{1-q}(t_0,x)dxdt.
\end{split}
\end{equation*}
Using these estimates in \eqref{mt2_24}, as $r \to 0^+$ we get
\begin{equation}\label{mt2_25}
\begin{split}
& \int_{t_0}^{t_1} \int_\Om \left( \frac{\partial u}{\partial t}\right)^2 dxdt+ \frac{1}{2}\|u(t_1,\cdot)\|_{X_0(\Om)}^2-\frac{1}{1-q}\int_\Om u(t_1)^{1-q}dx\\
&\quad \geq \int_{t_0}^{t_1} \int_\Om h\frac{\partial u}{\partial t}~ dxdt+ \frac{1}{2}\|u(t_0,\cdot)\|_{X_0(\Om)}^2-\frac{1}{1-q}\int_\Om u(t_0)^{1-q}dx.
\end{split}
\end{equation}
The inequality \eqref{mt2_25} along with \eqref{mt2_22} gives the equality. Since the map $t \mapsto \int_\Om u^{1-q}(t,x)dt$ is continuous, $u \in C([0,T]; X_0(\Om))$. Also, \eqref{mt2_eq} is obtained by taking $t_1=t \in [0,T]$ and $t_0=0$.\QED
\end{proof}

\noi \textbf{Proof of Theorem \ref{mt2}:} The proof follows from Theorem \ref{mt2-prt1} and Theorem \ref{mt2-prt2}.\QED

Next, we present the proof of Proposition \ref{para_sing_prop1} and end this section. Through this Proposition, the solution obtained above for $(G^s_t)$ can be proved to belong in $C([0,T]; C_0(\overline{\Om}))$ if the initial function $u_0 \in \overline{\mc D(L)}^{L^\infty}$.
in Section 2.\\

\noi \textbf{Proof of Proposition \ref{para_sing_prop1}: } Let $u_0 \in \overline{\mc D(L)}^{L^\infty}$.
Let $\la>0$ and $f_1,f_2 \in L^\infty(\Om)$. Let $u,v \in X_0(\Om)\cap \mc C\cap C_0(\overline{\Om})$ be the unique solution to
\begin{equation}\label{prop1_1}
\left\{
\begin{split}
u+\la L(u) &= f_1 \; \text{in}\; \Om,\\
v+\la L(v) &= f_2 \; \text{in}\; \Om,
\end{split}
\right.
\end{equation}
as obtained using Theorem \ref{mt1}. Then obviously, $u,v \in \mc D(L)$. We define $w:= (u-v-\|f_1-f_2\|_{\infty})^+$ and taking $w$ as test function, from \eqref{prop1_1} we get
\begin{equation}\label{prop1_2}
\int_\Om w^2dx +\la \int_\Om (L(u)-L(v))w~dx \leq 0.
\end{equation}
It is easy to compute that $\displaystyle\int_\Om (L(u)-L(v))w~dx \geq 0$.
So if $\text{supp}(w)$ has nonzero measure, then
\begin{equation*}
\int_\Om w^2dx +\la \int_\Om (L(u)-L(v))w~dx > 0
\end{equation*}
which contradicts \eqref{prop1_2}. Therefore $(u-v)\leq \|f_1 -f_2\|_\infty$ and if we reverse the roles of $u$ and $v$ then we get $\|u-v\|_\infty \leq \|f_1 -f_2\|_\infty$. This proves that $L$ is m-accretive in $L^\infty(\Om)$.
 Let $\tilde w \in \mc D(L)$ and $a,b \in L^\infty(\Lambda_T)$. Then further proof of Proposition \ref{para_sing_prop1} can be obtained using Chapter $4$, Theorem $4.2$ and Theorem $4.4$ of \cite{Barbu} or following proof of Proposition $0.1$ of \cite{bbg}.\QED

\section{Existence of unique solution to $(Q^s)$}
 We give the proof of Theorem \ref{mt4} in this section. Before doing that, we prove a weak comparison principle which is needed to prove Theorem \ref{mt4}. We recall the following discrete Picone identity which will be required to prove the weak comparison principle.
\begin{Lemma}(Lemma $6.2$, \cite{Pi-id})\label{Picone-id}
Let $p\in (1,+\infty)$. For $u,v : \Om \subset \mb R^n \to \mb R$ such that $u\geq0$, $v> 0$, we have
\[M(u,v)\geq 0\; \text{in}\; \mb R^n \times \mb R^n,\]
where $\ds M(u,v)= |u(x)-u(y)|^p-|v(x)-v(y)|^{p-2}(v(x)-v(y))\left(\frac{u(x)^p}{v(x)^{p-1}}- \frac{u(y)^p}{v(y)^{p-1}} \right).$\\
The equality holds in $\Om$ if and only if $u=kv$ a.e. in $\Om$, for some constant $k$.
\end{Lemma}

\begin{Theorem}\label{wcp}
Let $g:\Om\times \mb R^+\to \mb R$ be a Carath\'{e}odary function bounded below such that the map $y \mapsto \frac{g(x,y)}{y}$ is decreasing in
$\mb R^+$ for a.e. $x \in \Om$. Let $u,v \in L^\infty(\Om)\cap X_0(\Om)$ be such that $u,v>0$ in $\Om$,
 \begin{equation}\label{wcp_n}
 \int_{\Om}u^{1-q}~dx <+\infty, \;\int_{\Om}v^{1-q}~dx < +\infty
\end{equation}
 and  satisfies
\begin{equation}\label{wcp1}
\begin{split}
(-\De)^s u \leq \frac{1}{u^q}+ g(x,u)\; \text{and}\;(-\De)^s v \geq \frac{1}{v^q}+ g(x,v) \; \text{weakly in }\; (X_0(\Om))^*.
\end{split}
\end{equation}
Moreover, if there exists $0<w\in L^\infty(\Om)$ such that $c_1 w \leq u,v\leq c_2 w$, for $c_1,c_2>0$ constants and
\begin{equation}\label{ldct}
\int_\Om |g(x,c_1w)|w~dx < +\infty,\; \; \int_\Om |g(x,c_2w)|w~dx < +\infty,
\end{equation}
then $u\leq v$ in $\Omega$.
\end{Theorem}
\begin{proof}
For $k>0$, let us define $u_k := u+\frac{1}{k}$ and $v_k := v+\frac{1}{k}$. Also let
\[\phi_k := \frac{u_k^2-v_k^2}{u_k}\; \; \; \text{and}\;\; \; \psi_k := \frac{v_k^2-u_k^2}{v_k}.\]
Since $u,v\in L^\infty(\Om)$, obviously $u_k,v_k \in L^\infty(\Om)$ and thus $u_k,v_k \in L^2(\Om)$. {We assumed $u,v \in X_0(\Om)$, this implies $u,v \in H^s(\Om)$. Since $\|u_k\|_{H^s(\Om)}= \|u\|_{H^s(\Om)}$ and $\|v_k\|_{H^s(\Om)}= \|v\|_{H^s(\Om)}$ we conclude that $u_k,v_k \in H^s(\Om)$. Let
\[\eta_k := \frac{v_k^2}{u_k} \; \text{and}\; \xi_k:= \frac{u_k^2}{v_k} \]
then we claim that $\eta_k, \xi_k \in H^s(\Om)$. Consider
\begin{equation}
\begin{split}
|\eta_k(x)-\eta_k(y)| &= \left| \frac{v_k^2(x)- v_k^2(y)}{u_k(x)}- \frac{v_k^2(y)(u_k(x)-u_k(y))}{u_k(x)u_k(y)}\right|\\
&\leq k|v_k(x)-v_k(y)||v_k(x)+v_k(y)|+ \|v_k\|_{L^{\infty}(\Om)}^2 \frac{|u_k(x)-u_k(y)|}{u_k(x)u_k(y)}\\
& \leq 2k\|v_k\|_{L^{\infty}(\Om)}|v_k(x)-v_k(y)| + k^2 \|v_k\|_{L^{\infty}(\Om)}^2 |u_k(x)-u_k(y)|\\
& \leq C(k,\|v_k\|_{L^{\infty}(\Om)})(|v_k(x)-v_k(y)| +|u_k(x)-u_k(y)|),
\end{split}
\end{equation}
where $ C(k,\|v_k\|_{L^{\infty}(\Om)})>0$ is a constant. Since $u_k,v_k \in H^s(\Om)$, $\eta_k \in H^s(\Om)$. Similarly $\xi_k\in H^s(\Om)$. Clearly, this implies that $\phi_k,\psi_k \in H^s(\Om)$. We note that $\phi_k, \psi_k$ can also be written as
\[\phi_k= \frac{(u-v)(u_k+v_k)}{u_k}\; \text{and}\; \psi_k= \frac{(v-u)(v_k+u_k)}{v_k}\]
which implies that $\phi_k,\psi_k = 0$ in $\mb R^n \setminus \Om$ i.e. $\phi_k,\psi_k \in X_0(\Om)$ since $\frac{u_k+v_k}{u_k}$ and $\frac{u_k+v_k}{v_k}$ in $L^\infty(\Om)$.}
We set $\Om^+=\{x\in \Om: u(x)>v(x)\}$ and $\Om^-=\{x\in \Om: u(x)\leq v(x)\}$. Then $\phi_k \geq 0$ and $\psi_k \leq 0$ in $\Om^+$. Let $\tilde \phi_k = \chi_{\Om^+}\phi_k$  and $\tilde \psi_k =  \chi_{\Om^+}\psi_k$. Since  $\phi_k(x)\leq \phi_k(x)-\phi_k(y)$ for $(x,y)\in \Om^+\times \Om^-$, we get
\begin{align*}
&\int_Q \frac{|\tilde \phi_k(x)- \tilde \phi_k(y)|^2}{|x-y|^{n+2s}}~dxdy\\
&= \int_{\Om^+}\int_{\Om^+} \frac{| \phi_k(x)-  \phi_k(y)|^2}{|x-y|^{n+2s}}~dxdy+ 2\int_{\Om^+}\int_{\Om^-} \frac{| \phi_k(x)|^2}{|x-y|^{n+2s}}~dxdy+ 2 \int_{\Om^+}\int_{\mc C \Om}  \frac{| \phi_k(x)|^2}{|x-y|^{n+2s}}~dxdy \\
&\leq \int_{\Om^+}\int_{\Om^+} \frac{| \phi_k(x)-  \phi_k(y)|^2}{|x-y|^{n+2s}}~dxdy+2\int_{\Om^+}\int_{\Om^-} \frac{| \phi_k(x)- \phi_k(y)|^2}{|x-y|^{n+2s}}~dxdy\\
&\quad  + 2 \int_{\Om}\int_{\mc C \Om}  \frac{| \phi_k(x)|^2}{|x-y|^{n+2s}}~dxdy = \|\phi_k\|_{X_0(\Om)}^2<+\infty.
\end{align*}
This implies $\tilde \phi_k \in X_0(\Om)$ since by definition $\tilde \phi_k =0 $ in $\mb R^n\setminus \Om$. Similarly, $\tilde \psi_k \in X_0(\Om)$.
 Using $\tilde \phi_k$ and $\tilde \psi_k$ as test functions in \eqref{wcp1}, we get
\begin{equation}\label{wcp2}
\begin{split}
\int_{\mb R^n}((-\De)^s u)\tilde \phi_k~dx &\leq \int_{\Om^+}\left(\frac{1}{u^q}+ g(x,u)\right)\phi_k~dx,\\
\int_{\mb R^n}((-\De)^s v)\tilde \psi_k~dx &\leq \int_{\Om^+}\left(\frac{1}{v^q}+ g(x,v)\right)\psi_k~dx.
\end{split}
\end{equation}
Consider
\begin{equation}\label{wcp3}
\begin{split}
&\int_{\Om^+}\int_{\Om^+} \frac{(u(x)-u(y))(\phi_k(x)-\phi_k(y))}{|x-y|^{n+2s}}~dxdy+ \int_{\Om^+}\int_{\Om^+} \frac{(v(x)-v(y))(\psi_k(x)-\psi_k(y))}{|x-y|^{n+2s}}~dxdy\\
& =  \int_{\Om^+}\int_{\Om^+}\frac{(u_k(x)-u_k(y))^2 +(v_k(x)-v_k(y))^2}{|x-y|^{n+2s}}~dxdy\\
&\quad \quad-  \int_{\Om^+}\int_{\Om^+}\frac{\left((v_k(x)-v_k(y))\left( \frac{u_k^2(x)}{v_k(x)}-\frac{u_k^2(y)}{v_k(y)}\right)+(u_k(x)-u_k(y))\left( \frac{v_k^2(x)}{u_k(x)}-\frac{v_k^2(y)}{u_k(y)}\right) \right)}{|x-y|^{n+2s}}~dxdy\\
& = \int_{\Om^+}\int_{\Om^+}\frac{M(u_k,v_k)+ M(v_k,u_k)}{|x-y|^{n+2s}}~dxdy\geq 0,
\end{split}
\end{equation}
using Lemma \ref{Picone-id} with $p=2$. We have
\[\int_{\Om^+}\left(\frac{ \phi_k}{u^q}+ \frac{ \psi_k}{v^q}\right)~dx \leq 0.\]
Using this, we get
\begin{equation}\label{wcp4}
\begin{split}
&\int_{\Om^+}\left( \frac{1}{u^q}+g(x,u)\right)\phi_k~dx+ \int_{\Om^+}\left( \frac{1}{v^q}+g(x,v)\right)\psi_k~dx\\
& \leq \int_{\Om^+}(g(x,u)\phi_k+ g(x,v)\psi_k)~dx= \int_{\Om^+}\left(\frac{g(x,u)}{u}\left(\frac{u}{u_k}\right)- \frac{g(x,v)}{v}\left(\frac{v}{v_k}\right)\right)(u_k^2-v_k^2)~dx.
\end{split}
\end{equation}
Since $\frac{u}{u_k}\to 1$ and $\frac{v}{v_k}\to 1$ a.e. in $\Om$ as $k\to +\infty$, using \eqref{ldct} and Lebesgue Dominated convergence theorem with \eqref{wcp4} we get $\lim\limits_{k \to +\infty}{\int_{\Om^+}}(g(x,u)\phi_k+g(x,v)\psi_k)~dx=0$. Therefore  \eqref{wcp4} implies that
\begin{equation}\label{wcp5}
\lim_{k\to +\infty}\int_{\Om^+ }\left( \frac{1}{u^q}+g(x,u)\right)\phi_k~dx+ \int_{\Om^+}\left( \frac{1}{v^q}+g(x,v)\right)\psi_k~dx \leq 0.
\end{equation}
From \eqref{wcp2}, we have that
\begin{equation}\label{wcp6}
\int_{\Om^+}(((-\De)^s u)\phi_k+((-\De)^s v)\psi_k)~dx \leq \int_{\Om^+}\left(\left(\frac{1}{u^q}+ g(x,u)\right)\phi_k+\left(\frac{1}{v^q}+ g(x,v)\right)\psi_k\right)~dx,
\end{equation}
We claim that
\begin{equation}\label{wcp7}
\begin{split}
&\int_Q \frac{(u(x)-u(y))(\tilde \phi_k(x)-\tilde \phi_k(y))}{|x-y|^{n+2s}}~dxdy+ \int_Q \frac{(v(x)-v(y))(\tilde \psi_k(x)-\tilde \psi_k(y))}{|x-y|^{n+2s}}~dxdy\\
 &\geq \int_{\Om^+}\int_{\Om^+} \frac{(u(x)-u(y))( \phi_k(x)- \phi_k(y))+ (v(x)-v(y))(\psi_k(x)-\psi_k(y))}{|x-y|^{n+2s}}~dxdy
 %+ \int_{\Om^+}\int_{\Om^+} \frac{(v(x)-v(y))(\psi_k(x)-\psi_k(y))}{|x-y|^{n+2s}}~dxdy.
\end{split}
\end{equation}
To prove this we consider
\begin{align*}
&\int_Q \frac{(u(x)-u(y))(\tilde \phi_k(x)-\tilde \phi_k(y))}{|x-y|^{n+2s}}~dxdy+ \int_Q \frac{(v(x)-v(y))(\tilde \psi_k(x)-\tilde \psi_k(y))}{|x-y|^{n+2s}}~dxdy\\
&= \int_{\Om^+}\int_{\Om^+} \frac{(u(x)-u(y))( \phi_k(x)- \phi_k(y)) + (v(x)-v(y))(\psi_k(x)-\psi_k(y))}{|x-y|^{n+2s}}~dxdy\\
& \quad +2 \int_{\Om^+}\int_{\Om^-}\frac{(u(x)-u(y))\phi_k(x)+ (v(x)-v(y))\psi_k(x)}{|x-y|^{n+2s}}~dxdy \\
&\quad \quad + 2 \int_{\Om^+}\int_{\mc C \Om} \frac{(u(x)-u(y))\phi_k(x)+ (v(x)-v(y))\psi_k(x)}{|x-y|^{n+2s}}~dxdy.
\end{align*}
Since $\phi_ku_k+\psi_kv_k =0$ by definition and $\phi_k+\psi_k \leq 0$ in $\Om^+$ and $\Om^-$ both, we get
\begin{align*}
&\int_{\Om^+}\int_{\Om^-}\frac{(u(x)-u(y))\phi_k(x)+ (v(x)-v(y))\psi_k(x)}{|x-y|^{n+2s}}~dxdy\\
&= \int_{\Om^+}\int_{\Om^-}\frac{(u_k(x)-u_k(y))\phi_k(x)+ (v_k(x)-v_k(y))\psi_k(x)}{|x-y|^{n+2s}}~dxdy\\
& = - \int_{\Om^+}\int_{\Om^-}\frac{u_k(y)\phi_k(x)+ v_k(y)\psi_k(x)}{|x-y|^{n+2s}}~dxdy \geq - \int_{\Om^+}\int_{\Om^-}\frac{v_k(y)(\phi_k(x)+\psi_k(x))}{|x-y|^{n+2s}}~dxdy\geq 0.
\end{align*}
Similarly
\begin{align*}
&\int_{\Om^+}\int_{\mc C \Om} \frac{(u(x)-u(y))\phi_k(x)+ (v(x)-v(y))\psi_k(x)}{|x-y|^{n+2s}}~dxdy\\
 & \quad \quad=\int_{\Om^+}\int_{\mc C \Om} \frac{-(\phi_k(x)+\psi_k(x))}{k|x-y|^{n+2s}}~dxdy \geq 0.
\end{align*}
This establishes our claim.
Therefore using \eqref{wcp3}, \eqref{wcp5}, \eqref{wcp6}, \eqref{wcp7} and Fatou's Lemma , we get
\begin{equation*}
\begin{split}
0 & \leq \int_{\Om^+}\int_{\Om^+}\frac{M(u,v)+ M(v,u)}{|x-y|^{n+2s}}~dxdy\leq \lim_{k \to +\infty}\left(\int_{\mb R^n}((-\De)^su)\tilde \phi_k~dx + \int_{\mb R^n}((-\De)^sv)\tilde \psi_k~dx\right)\\
&\leq  \lim_{k \to +\infty}\int_{\Om^+}\left(\left(\frac{1}{u^q}+ g(x,u)\right)\phi_k+\left(\frac{1}{v^q}+ g(x,v)\right)\psi_k\right)~dx \leq 0.
\end{split}
\end{equation*}
{This implies that}
\[\int_{\Om^+}\int_{\Om^+}\frac{M(u,v)+M(v,u)}{|x-y|^{n+2s}}~dxdy =0.\]
{Therefore} $M(u,v)=0= M(v,u)$ a.e. in $\Om^+$.  So using Lemma \ref{Picone-id}, we have $u = kv$ a.e. in $\Om^+$, for some constant $k>0$. By definition of $\Om^+$, we have $k>1$. Consider
\begin{equation}\label{wcp8}
\begin{split}
&\int_{\Om^+}(((-\De)^su)u-((-\De)^s kv)kv)~dx = \int_{\Om^+}((-\De)^su- (-\De)^s kv)kv)~dx\\
& = \int_{\Om^+} ( (-\De)^s(u-kv))kv~dx =2C_n^s\int_{\Om^+}\left( P.V.\int_{\mb R^n}\frac{(u-kv)(x)- (u-kv)(y)}{|x-y|^{n+2s}}dy\right)kv(x)dx\\
& = 2C_n^s \int_{\Om^+}P.V.\int_{\Om^-}\frac{(kv-u)(y)}{|x-y|^{n+2s}}kv(x)dxdy \geq 2C_n^s k^2 \int_{\Om^+}P.V.\int_{\Om^-}\frac{(v-u)(y)}{|x-y|^{n+2s}}v(x)dxdy\geq 0.
\end{split}
\end{equation}
From \eqref{wcp_n} and \eqref{wcp1} we get
\begin{equation}
\begin{split}
\int_{\Om^+}((-\De)^su)u~dx &\leq \int_{\Om^+} \left(\frac{g(x,kv)}{kv}(kv)^2 + k^{1-q}v^{1-q}\right)~dx\; \text{and}\;\\
k^2 \int_{\Om^+}((-\De)^sv)v~dx & \geq \int_{\Om^+} \left(\frac{g(x,v)}{v}(kv)^2 + k^2 v^{1-q}\right)dx
\end{split}
\end{equation}
which implies that $k\leq 1$ by \eqref{wcp8}. This gives a contradiction which implies $u\leq v$ in $\Om$. \hfill{\QED}
\end{proof}
%%%%%%%%%%%%%%%%%%%%%%%%%%%%%%%%%%%%%%%%%%%%%%%%%%%%%%%%%%%%%%%%%%%%%%%%%%%%%%%%%%%%%%%%%%%%%%%%%%%%%%%%%%%%%%%%%%%%%%%%%%

 \noi \textbf{Proof of Theorem \ref{mt4}: } Under the hypothesis on $f$, we let $l,\mu>0$ be such that $-l \leq f(x,y)\leq \mu y+l$. Let $\mu$ be such that $0<  \mu < \la_{1,s}(\Om)$.
 Suppose $w$ is a solution of  \eqref{sp_mt1_4}. For $\eta>0$, we define
 \begin{equation}\label{mt4_1}
 \underline{u}= \eta w.
 \end{equation}
Since $w \in \mc C \cap C_0(\overline{\Om})$ (see \eqref{sp_mt1_5}-\eqref{sp_mt1_6}), we can choose $\eta>0$ small enough such that
\begin{equation}\label{mt4_3}
\begin{split}
(-\De)^s\underline{u} -\frac{1}{\underline{u}^q} &\leq -l \leq f(x,\underline{u}) \; \text{in}\; \Om,\; \; \underline{u}=0 \; \text{in}\; \mb R^n \setminus \Om.
%(-\De)^s\overline{u} -\frac{1}{\overline{u}^q} & \geq \mu \overline{u}+l \geq f(x,\overline{u}) \; \text{in}\; \Om,\; \; \overline{u}=0 \; \text{in}\; \mb R^n \setminus \Om.
\end{split}
\end{equation}
Let $0<M,M^\prime$ and
 \begin{equation}\label{mt4_2}
 \overline{u}= Mw+M^\prime \phi_{1,s}
 \end{equation}
 Let $\e>0$ and define $\Om_\e:= \{x \in \Om:\; \text{dist}(x,\partial \Om)<\e \}$. Since we know that $w=0$ in $\mb R^n \setminus \Om$, we can choose $\e>0$ small enough such that $0\leq w\leq c$ in $\Om_\e$ where $c>0$ is such that
 \[\left(M-\frac{1}{M^q}\right)\frac{1}{c^q} \geq \mu Mc+l\]
 which is possible for $c>0$ sufficiently small. Therefore in $\Om_\e$ we get
 \begin{equation}\label{mt4_4}
 \begin{split}
 (-\De)^s\overline{u}- \frac{1}{\overline{u}^q} &= \left(M-\frac{1}{M^q}\right)\frac{1}{w^q} + M^\prime \la_{1,s}\phi_{1,s}\geq \left(M-\frac{1}{M^q}\right)\frac{1}{c^q}+ M^\prime\mu\phi_{1,s} \\
 & \geq \mu Mc+l + M^\prime \mu \phi_{1,s} \geq \mu Mw +l + M^\prime \mu\phi_{1,s}= \mu \overline{u}+l.
 \end{split}
 \end{equation}
 Now consider the set $\Om \setminus \Om_\e= \{x \in \Om:\; d(x,\partial \Om)\geq \e\}$. Then there exists a constant $c_1>0$ (depending on $\e$) such that $0< c_1\leq \phi_{1,s}$ in $\Om\setminus \Om_\e$. Since $\mu < \la_{1,s}$ and $M$ is fixed now, we choose
 \[M^\prime \geq \frac{\mu M\|w\|_\infty+l}{c_1(\la_{1,s}-\mu)}.\]
 Then in $\Om \setminus \Om_\e$ we get
  \begin{equation}\label{mt4_7}
 \begin{split}
 (-\De)^s\overline{u}- \frac{1}{\overline{u}^q} = \left(M-\frac{1}{M^q}\right)\frac{1}{w^q} + M^\prime \la_{1,s}\phi_{1,s}\geq M^\prime\la_{1,s}\phi_{1,s}  \geq \mu Mw +l + M^\prime \mu\phi_{1,s}= \mu \overline{u}+l.
 \end{split}
 \end{equation}
 Therefore \eqref{mt4_4} and \eqref{mt4_7} implies that $\overline{u}$ satisfies
 \begin{equation}\label{mt4_9}
 \begin{split}
 (-\De)^s\overline{u} -\frac{1}{\overline{u}^q} & \geq \mu \overline{u}+l \geq f(x,\overline{u}) \; \text{in}\; \Om,\; \; \overline{u}=0 \; \text{in}\; \mb R^n \setminus \Om.
\end{split}
\end{equation}
By construction, $\underline{u}, \overline{u}\in \mc C$. Since $f$ uniformly locally lipschitz with respect to second variable, we can find appropriate constant $K_0>0$ such that the map $t \mapsto K_0 t + f(x,t)$ is non-decreasing in $[0,\|\overline{u}\|_{X_0(\Om)}]$, for a.e. $x \in \Om$. We define an iterative scheme to obtain a sequence $\{u_k\} \subset X_0(\Om)\cap \mc C\cap C_0(\overline{\Om})$ (using Theorem \ref{mt2}) as solution to the problem
\begin{equation}\label{mt4_5}
\left\{
\begin{split}
(-\De)^s u_k -\frac{1}{u_k^q}+K_0 u_k = f(x,u_{k-1})+K_0 u_{k-1}\; \text{in}\; \Om,\;\;u_k =0, \; \text{in}\; \mb R^n\setminus \Om,
\end{split}
\right.
\end{equation}
where $u_0:= \underline{u}$. This scheme is well defined because by the choice of $K_0$ and using weak comparison principle (Lemma \ref{weak_comp_principle}), we get that
\begin{equation}\label{mt4_6}
{\underline{u} \leq} u_k \leq \overline{u},
\end{equation}
 for all $k$. This implies for each $k$, right hand side of \eqref{mt4_5} is in $L^\infty(\Lambda_T)$ and hence Theorem \ref{mt2} is applicable for \eqref{mt4_5}. Again using Lemma \ref{weak_comp_principle} and monotonicity of the map $t \mapsto K_0 t + f(x,t)$, we have that the sequence $\{u_k\}$ is a monotone increasing sequence. From \eqref{mt4_5} we have $(\De)^su_k = g_k \in L^\infty(\Om^\prime)$, where $g_k := u_k^{-q}- K_0 u_k + f(x,u_{k-1})+K_0 u_{k-1} \leq \underline{u}^{-q}-K_0 \underline{u}+ f(x,\overline{u})+ K_0 \overline{u}$ and $\Om^\prime$ is a compact subset of $\Om$. Following proof of Theorem $1.2$ of \cite{aJs}, we get that $u_k \in C^{s-\e}(\mb R^n)$ for each $\e>0$ small enough when $q=1$ and $u_k\in C^{\frac{2s}{q+1}}(\mb R^n)$ when $q>1$. Also since \eqref{mt4_6} holds, we get that $\{u_k\}$ is a uniformly bounded sequence in $C_0(\overline{\Om})\cap \mc C$. {Therefore by Arzela Ascoli theorem we know that there exist} $\tilde u \in C_0(\overline{\Om})\cap \mc C$ such that $u_k \uparrow \tilde u$ in $C_0(\overline{\Om})\cap \mc C$ as $k \to \infty$. Therefore it must be Cauchy in $C_0(\overline{\Om})\cap \mc C$ and this alongwith \eqref{mt4_5} gives that $\{u_k\}$ is Cauchy in $X_0(\Om)$ which converges to $\tilde u$ in $X_0(\Om)$. Now passing on to the limits as $k\to \infty$ and using Lebesgue Dominated convergence theorem (since $u_k\leq \overline{u}$, for all $k$) in \eqref{mt4_5}, we obtain $\tilde u$ to be solution to $(Q^s)$. Lastly, uniqueness of $\tilde u$ follows from Theorem \ref{wcp}.\QED

\section{Existence of solution to $(P^s_t)$ and its regularity}
We devote this section to study the problem $(P^s_t)$ which is our concern for this article. Precisely, we will prove Theorem \ref{mt3} and Proposition \ref{para_sing_prop2}.\\
\noi \textbf{Proof of Theorem \ref{mt3}: } We will closely make use of arguments in the proof of Theorem \ref{mt2} while proving this theorem. Since $T>0$, we define $\De_t:= \frac{T}{n}$, where $n \in \mb N^*$. Taking $u^0= u_0$, we obtain a sequence $\{u^k\} \subset \mc C \cap X_0(\Om) \subset L^\infty(\Om)$ as solutions to following iterative scheme
\begin{equation}\label{mt3_1n}
\begin{split}
u^k- \De_t \left((-\De)^s u^k + \frac{1}{(u^k)^q} \right) = \De_t f(x,u^{k-1})+u^{k-1}\; \text{in}\; \Om.
\end{split}
\end{equation}
Since $u^0 \in \mc C\cap X_0(\Om)$ and $\De_t f(x,u^{k-1})+ u^{k-1} \in L^\infty(\Lambda_T)$ for each $k$, we can apply Theorem \ref{mt2} to obtain the sequence $\{u^k\} \subset \mc C \cap X_0(\Om) \subset L^\infty(\Om)$. In \eqref{mt4_1} and \eqref{mt4_2}, we can choose $\eta,M, M^\prime>0$ appropriately such that $\underline{u} \leq u_0 \leq \overline{u}$ (since $u_0 \in \mc C$). Using $-l \leq f(x,y)\leq \mu y +l $ and applying Lemma \ref{weak_comp_principle} iteratively, we can get $\underline u \leq u^k\leq \overline u$, for all $k$. We remark that it is clear from definition in \eqref{mt4_1} that $\underline u$ and  $\overline u$ are independent of $\De_t$. Let $u_{\De_t}$ and $\tilde u_{\De_t}$  be as defined in \eqref{mt2_2} alongwith the assumption that $u_{\De_t}(t)= u_0$, when $t<0$. Then it is easy to see that \eqref{mt2_3} is satisfied with $h_{\De_t}(t,x):= f(x,u_{\De_t}(t-\De_t,x))$, for $t \in [0,T]$ and $x \in \Om$. Using \eqref{mt2_7}, we have $ \underline u \leq u_{\De_t} \leq \overline u$. Therefore,
\begin{equation*}
%\begin{split}
h_{\De_t}(t,x) \leq \mu u_{\De_t}(t-\De_t,x)+l \in L^\infty(\Lambda_T)
%\end{split}
\end{equation*}
independent of $\De_t$. Hence we can use similar techniques as in proof of Theorem \ref{mt2} to get
\begin{equation}\label{mt3_1}
\begin{split}
&u_{\De_t}, \tilde u_{\De_t} \in L^\infty ([0,T]; X_0(\Om)\cap \mc C), \; u_{\De_t}, \tilde u_{\De_t}  \in L^\infty(\Lambda_T),\frac{\partial \tilde u_{\De_t}}{\partial t}  \in L^2(\Lambda_T),\\
&\;\|u_{\De_t} - \tilde u_{\De_t}\|_{L^2(\Om)} \leq C(\De_t)^{\frac12}\; \text{and}\; \frac{1}{(u_{\De_t})^q}  \in L^\infty([0,T]; (X_0(\Om))^*)
\end{split}
\end{equation}
{uniformly in $\De_t$}. So we can use the Banach Alaoglu theorem and \eqref{mt3_1} to get $u \in L^\infty([0,T];X_0(\Om))$ and $u \in L^\infty (\Lambda_T)$ such that, upto a subsequence,
\begin{equation}\label{mt3_2}
\begin{split}
 u_{\De_t}, \tilde u_{\De_t} \xrightharpoonup{\text{*}} L^\infty([0,T]; X_0(\Om))\; \text{and in}\; L^\infty(\Lambda_T),\;\frac{\partial \tilde u_{\De_t}}{\partial t} \rightharpoonup \frac{\partial u}{\partial t} \; \text{in}\; L^2(\Lambda_T)
\end{split}
\end{equation}
as $\De_t \to 0^+$. Also similar to proof of Theorem \ref{mt2}, we get
\begin{equation}\label{mt3_2}
u_{\De_t}, \; \tilde u_{\De_t} \to u \; \text{in}\; L^\infty([0,T];L^2(\Om))\; \text{and}\; u \in C([0,T];L^2(\Om)).
\end{equation}
In addition, if $M>0$ denotes the Lipschitz constant for $f$ then for $t \in [0,T]$
\begin{equation}\label{mt3_3}
\begin{split}
\|h_{\De_t}(t,\cdot)- f(\cdot,u(t,\cdot))\|_{L^2(\Om)}&= \|f(\cdot, u_{\De_t}(t-\De_t,\cdot))- f(\cdot,u(t,\cdot))\|_{L^2(\Om)}\\
 &\leq M\|u_{\De_t}(t-\De_t,\cdot)- u(t,\cdot)\|_{L^2(\Om)}.
\end{split}
\end{equation}
From \eqref{mt3_2} and \eqref{mt3_3}, we deduce that $h_{\De_t}(t,x) \to f(x,u(x))$ in $L^\infty([0,T]; L^2(\Om))$. Finally, following exactly the last part of the proof of Theorem \ref{mt2}, we can show that $u \in \mc A(\Lambda_T)$ and $u$ is a weak solution to $(P_t^s)$.

It remains to prove the uniqueness. For that, let $v\in \mc A(\Lambda_T)$ be another weak solution to $(P^s_t)$. For fix $t_0 \in [0,T]$ we have
\begin{equation}\label{mt3_4}
\begin{split}
&\int_0^{t_0}\int_{\Om} \frac{\partial(u-v)}{\partial t}(u-v)~dxdt+ \int_0^{t_0}\int_{\mb R^n}((-\De)^s(u-v))(u-v)~dxdt\\
& \quad - \int_0^{t_0}\int_{\Om} \left( \frac{1}{u^q}-\frac{1}{v^q}\right)(u-v)~dxdt= \int_0^{t_0}\int_{\Om}(f(x,u(x)- f(x,v(x))))(u-v)~dxdt.
\end{split}
\end{equation}
From \eqref{mt3_4}, $u(0,x)=v(0,x)=u_0(x)$ in $\Om$ and $f$ being locally Lipschitz uniformly in $\Om$, we get
\begin{equation}\label{mt3_5}
\begin{split}
&\frac12\|(u-v)(t_0)\|_{L^2(\Om)}+ \int_0^{t_0}\int_{\mb R^n}((-\De)^s(u-v))(u-v)~dxdt- \int_0^{t_0}\int_{\Om}\left( \frac{1}{u^q}-\frac{1}{v^q}\right)(u-v)~dxdt\\
&\leq M \int_0^{t_0}\int_{\Om}|u-v|^2~dxdt,
\end{split}
\end{equation}
where $M$ is Lipschitz constant for $f$. From Lemma \ref{weak_comp_principle}, we know that the operator $A$ is strictly monotone which gives
\begin{equation*}
\begin{split}
 0 & < \int_{0}^{t_0}\int_{\Om}|(u-v)|^2~dxdt+ \int_0^{t_0}\int_{\mb R^n}((-\De)^s(u-v))(u-v)~dxdt\\
 & \quad \quad-\int_0^{t_0}\int_{\Om} \left( \frac{1}{u^q}-\frac{1}{v^q}\right)(u-v)~dxdt.
\end{split}
\end{equation*}
Using this with \eqref{mt3_5}, we get
\[\frac12\|(u-v)(t_0)\|_{L^2(\Om)}\leq M_0\int_0^{t_0}\int_{\Om}|u-v|^2~dxdt, \]
where $M_0>0$ is a constant. By Gronwall's inequality, we get $\|(u-v)(t_0,\cdot)\|_{L^2(\Om)}\leq \|(u-v)(0,\cdot)\|_{L^2(\Om)}\exp(M_0t_0)$. Since $u(0,\cdot)=v(0,\cdot)$ and this holds for all $t_0\in [0,T]$, we get $u\equiv v$. This completes the proof.\QED

Now we give the proof of Proposition \ref{para_sing_prop2}.\\

\noi \textbf{Proof of Proposition \ref{para_sing_prop2}: } Using Proposition \ref{para_sing_prop1} above and following the proof of Proposition $0.2$ of \cite{bbg}, the result can be similarly obtained. \QED

\section{Asymptotic Behavior}
In this section, we present the proof of Theorem \ref{mt5}.

\noi \textbf{Proof of Theorem \ref{mt5}: } Let {$\underline{u}, \overline{u} \in  \mc C \cap X_0(\Om)\cap C_0(\overline \Om)$}, be the sub and supersolution respectively to
\begin{equation}\label{aymp1}
\left\{
\begin{split}
(-\De)^s u - \frac{1}{u^q}&= f(x,u)\; \text{in}\; \Om,\\
u &=0 \; \text{in} \; \mb R^n \setminus \Om,
\end{split}
\right.
\end{equation}
where $\underline{u}, \overline{u}$ is defined in \eqref{mt4_1}. We can choose $\eta >0$ small enough and $M>0$ large enough so that $\underline{u}\leq u_0 \leq \overline{u}$ which is possible because we took $u_0 \in \mc C \cap X_0(\Om)$. Let $u$ be solution of $(P_t^s)$ and $v_1$ and $ v_2$ be unique solutions to $(P^s_t)$ with initial datum $\underline u$ and $\overline u$. The existence of $v_1$ and $v_2$ are justified through Theorem \ref{mt3}. We claim that $\underline{u}, \overline{u}\in \overline{\mc D(L)}^{L^\infty(\Om)}$. Let $g,h \in (X_0(\Om))^{*}$ be functions such that $L(\underline u) = g$ and  $L(\overline u) = h$. Using \eqref{mt4_3}, we have $g\leq 0$ and $h \geq 0$. Now, let  $\{g_k\} = \max\{g, -k\}$, $\{h_k\}= \min\{h,k\}$ and $\{u_k\},\{w_k\}$ be two sequences in $\mc D(L)$ defined by $L(u_k)= g_k$, $L(w_k)=h_k$. Since $L$ is a monotone operator, as Lemma \ref{weak_comp_principle} we can show a similar kind of weak comparison principle concerning $L$. Using that, we can get $\{u_k\}$ is non increasing while $\{w_k\}$ is non decreasing. By definition of $g_k, h_k$, we can show that $g_k \to g$ and $h_k \to h$ in $(X_0(\Om))^*$ as $k \to \infty$. This implies $u_k \to \underline u$ and $w_k \to \overline u$ in $X_0(\Om)$ as $k \to \infty$. Therefore, upto a subsequence, $u_k \to \underline u$ and $w_k \to \overline u$ pointwise a.e. in $\Om$ as $k \to \infty$. Using Dini's theorem, we get $u_k \to \underline u$ and $w_k \to \overline u$ in $L^\infty(\Om)$ as $k \to \infty$. This proves our claim.

Now we can use Theorem \ref{mt3} and Proposition \ref{para_sing_prop2} to obtain $v_1,v_2 \in C([0,T];C_0(\overline{\Om}))$. Taking $\underline{u}^0= \underline u$(respectively $\overline{u}^0= \overline u$), we consider the sequence $\{\underline{u}^k\}$(respectively $\{\overline{u}^k\}$) which is non decreasing(respectively non increasing) as solutions to the iterative scheme given by \eqref{mt3_1n}, for $0<\De_t<1/M$ where $M$ denotes the Lipschitz constant of $f$ on $[\underline u, \overline u]$.  If the sequence $\{u^k\}$ denotes the one that is obtained in \eqref{mt3_1n}, then by the choice of $\De_t$ we can show that
\begin{equation}\label{asymp2}
\underline{u}^k \leq u^k \leq \overline{u}^k.
\end{equation}
Let $u$ denotes the weak solution of $(P^s_t)$ as obtained in proof of Theorem \ref{mt3}. We can follow the proof of Theorem \ref{mt3} and use \eqref{asymp2} to obtain
\begin{equation}\label{asymp3}
v_1(t)\leq u(t)\leq v_2(t).
\end{equation}
Consider the maps $t\mapsto v_1(t,x)$ and $t \mapsto v_2(t,x)$ which are non decreasing and non increasing respectively. Let $v_1(t)\to \tilde v_1$ and $v_2(t) \to \tilde v_2$ as $t \to \infty$. Now let $S(t)$ denotes the semigroup on $L^\infty(\Om)$ generated by the given evolution equation $u_t + \la L(u)=f(x,u)$. Then we know
\[\tilde v_1 = \lim_{t^\prime \to +\infty}S(t^\prime +t)(\underline u)= S(t)\lim_{t^\prime \to +\infty}S(t^\prime)(\underline u)= S(t)\lim_{t^\prime \to +\infty}v_1(t^\prime)= S(t)\tilde v_1 \]
and analogously, we obtain
\[\tilde v_2= S(t)\tilde v_1. \]
Then $\tilde v_1$ and $\tilde v_2$ are stationary solutions to $(P_t^s)$ i.e. solves $(Q^s)$. But by uniqueness of solution to $(Q^s)$ as shown in Theorem \ref{mt4}, we get $\tilde v_1 = \tilde v_2= \hat u \in C(\overline \Om)$. Therefore, by Dini's theorem we get
\[v_1(t) \to \hat u \; \text{and}\; v_2(t) \to \hat u \; \text{in} \; L^\infty(\Om)\; \text{as}\; t \to \infty.\]
Using \eqref{asymp3}, we conclude that $u(t) \to \hat u$ in $L^\infty(\Om)$ as $t \to \infty$.\QED

 \linespread{0.5}

\noindent {\bf Acknowledgements:} The authors were funded by IFCAM (Indo-French Centre for Applied Mathematics) UMI CNRS 3494 under the project "Singular phenomena in reaction diffusion equations and in conservation laws".


\begin{thebibliography}{21}

%\footnotesize

\bibitem{ap}  B. Abdellaoui,  M. Medina, I.  Peral and  A. Primo, {\it Optimal results for the fractional heat equation involving the Hardy potential},  Nonlinear Anal. 140 (2016), 166-207.

\bibitem{aJs} {Adimurthi, J. Giacomoni and S. Santra, {\it Positive solutions to a fractional equation with singular nonlinearity}, preprint available at https://arxiv.org/pdf/1706.01965.pdf}

\bibitem{ai} N. Alibaud and C. Imbert, {\it Fractional semi-linear parabolic equations with unbounded data}, Transactions of the American Mathematical Society, 361  (5) (2009), 2527-2566.

\bibitem{Pi-id} S. Amghibech, {\it On the discrete version of Picone's identity}, Discrete App Math, 156 (1) (2008), 1�10.

\bibitem{ags} B. Avelin, U. Gianazza and S. Salsa, {\it Boundary estimates for certain degenerate and singular parabolic equations}, Journal of the European Mathematical Society,  18 (2) (2016), 381-424.
\bibitem{bbg} M. Badra, K. Bal and J. Giacomoni, {\it A singular parabolic equation: Existence, stabilization}, J. Differential Equations, 252 (2012), 5042-5075.

\bibitem{Barbu} V. Barbu, {\it Nonlinear Differential Equations of Monotone types in Banach Spaces}, { Springer Monogr. Math.}, Springer, New York, 2010.

\bibitem{bbmp} B. Barrios, I. De Bonis, M. Medina and I. Peral, {\it Semilinear problems for the fractional laplacian with a singular nonlinearity}, Open Math., 13 (1) (2015), 390-407.

\bibitem{bg} B. Bougherara and J. Giacomoni, {\it Existence of mild solutions for a singular parabolic equation and stabilization}, Adv. Nonlinear Anal., 4 (2)  (2015), 123-134.

%\bibitem{banks} H.T. Banks, {\it Modelling and Control in the Biomedical Sciences}, Lect. Notes Biomath., vol. 6, Springer-Verlag, Berlin,1975.

\bibitem{cf} L. Cafarelli and A. Figalli, {\it Regularity of solutions to the parabolic fractional obstacle problem}, Journal f\"{u}r die reine und angewandte Mathematik (Crelles Journal), 680 (2013), 191-233.

\bibitem{dm} J. D\'{a}vila and M. Montenegro, {\it Existence and asymptotic behavior for a singular parabolic equation}, Transactions of the
American Mathematical Society, 357  (5) (2004), 1801-1828.

\bibitem{Quass} L.M. Del Pezzo and  A. J. Quaas, {\it Non-resonant Fredholm alternative and anti-maximum principle for the fractional $p$-Laplacian}, Journal of Fixed Point Theory and Applications, 19 (1) (2017), 939-958.

%\bibitem{diaz}J.I. D\'{i}az, {\it Nonlinear Partial Differential Equations and Free Boundaries, vol. I: Elliptic Equations}, Res. Notes Math.,vol. 106, Pitman(Advanced Publishing Program), Boston, MA, 1985.

\bibitem{fk} A. Fino and G. Karch, {\it Decay of mass for nonlinear equation with fractional Laplacian}, Monatshefte f\"{u}r Mathematik, 160 (4) (2010), 375-384.

\bibitem{fm} G. Fragnelli and D. Mugnai, {\it Carleman estimates for singular parabolic equations with interior degeneracy and non-smooth coefficients}, Adv. Nonlinear Anal., 6 (1) (2017), 61-84.

\bibitem{hardy} R.L. Frank and R. Seiringer, {\it Non-linear ground state representations and sharp Hardy inequalities},
Journal of Functional Analysis, 255 (12) (2008), 3407-3430.

\bibitem{TJS} J. Giacomoni, T. Mukherjee and K. Sreenadh, {\it Positive solutions of fractional elliptic equation with critical and singular nonlinearity},   Adv. Nonlinear Anal.,  6 (3) (2016), 327-354.

\bibitem{kimlee} S. Kim and Ki-Ahm Lee, {\it H\"{o}lder estimates for singular non-local parabolic equations}, Journal of Functional Analysis, 261 (12) (2011), 3482-3518.

\bibitem{tp} T. Leonori, I.  Peral, A. Primo and F.  Soria, {\it Basic estimates for solutions of a class of nonlocal elliptic and parabolic equations}, Discrete Contin. Dyn. Syst. 35 (12) (2015), 6031-6068.

\bibitem{TS1} T. Mukherjee and K. Sreenadh, {\it Fractional elliptic equations with critical growth and singular nonlinearities}, Electronic Journal of Differential Equations,  54 (2016), 1-23.

\bibitem{Roston-serra} X. Ros-Oton and J. Serra, {\it  The Dirichlet problem for the fractional Laplacian: Regularity up to the boundary}, Journal de Math\'{e}matiques Pures et Appliqu\'{e}es,  101 (3)  (2014), 275-302.

\bibitem{fraceigen} R. Servadei, E. Valdinoci, {\it The Brezis-Nirenberg result for the fractional laplacian}, Transactions of the American Mathematical Society, 367 (1) (2015), 67-102.
\bibitem{varmethod} R. Servadei and E. Valdinoci, {\it Variational methods for non-local operators of elliptic type},
Discrete Contin. Dyn. Syst., 33 (5) (2013), 2105-2137.
\bibitem{Silvestre} L. Silvestre, {\it Regularity of the obstacle problem for a fractional power of the laplace operator}, Comm. Pure Appl. Math., 60 (1) (2007), 67-112.

\bibitem{aubin-simon} J. Simon, {\it Compact sets in the space $L^p(0, T ; B)$}, Ann. Mat. Pura Appl., 146 (1987), 65-96.


\bibitem{vazquez1} J.L. V\'{a}zquez, {\it  Nonlinear Diffusion with Fractional Laplacian Operators}, Nonlinear Partial Differential
Equations: the Abel Symposium 2010, Holden, Helge and Karlsen, Kenneth H. eds., Springer, 2012, 271-298.

\bibitem{vazquez2}  J.L. V\'{a}zquez, {\it Recent progress in the theory of Nonlinear Diffusion with Fractional Laplacian Operators}, Discrete and Continuous Dynamical Systems - Series S, 7 (4) (2014), 857-885.


\end{thebibliography}
\end{document}